\documentclass[12pt]{amsart}
\usepackage{amsmath,amsthm,amssymb}

\usepackage[utf8]{inputenc} 
\usepackage[T1]{fontenc}    
\usepackage{hyperref}       
\usepackage{url}           
\usepackage{booktabs}       
\usepackage{amsfonts}       
\usepackage{nicefrac}       
\usepackage{microtype}      
\usepackage{xcolor}

\usepackage[margin=1in]{geometry}
\usepackage{graphicx}
\usepackage{comment}
\usepackage{enumitem}
\setlist[itemize]{leftmargin=*}

\theoremstyle{plain}

\newtheorem{theorem}{Theorem}[section]
\newtheorem{lemma}[theorem]{Lemma}
\newtheorem{corollary}[theorem]{Corollary}
\newtheorem{proposition}[theorem]{Proposition}
\newtheorem{assumption}[theorem]{Assumption}
\newtheorem{conjecture}[theorem]{Open Problem}
\newtheorem{definition}[theorem]{Definition}

\theoremstyle{remark}
\newtheorem{remark}[theorem]{Remark}
\newtheorem{example}{Example}[subsection]

\newtheorem{problem}[theorem]{Open Problem}

\numberwithin{equation}{section}
\allowdisplaybreaks

\title[Generalized Elephant Random Walks]{Asymptotic Properties of \\ Generalized Elephant Random Walks} 

\author[K. Maulik]{Krishanu Maulik}
\address{
    Theoretical Statistics and Mathematics Unit \\
    Indian Statistical Institute \\
    203 B. T. Road, Kolkata 700108 \\
    West Bengal, India 
}
\email{kmisical@gmail.com}
\author[P. Roy]{Parthanil Roy}
\address{
    Department of Mathematics \\
   Indian Institute of Technology Bombay \\
    Maharashtra 400076 \\
    India 
}
\email{parthanil.roy@gmail.com}
\author[T. Sadhukhan]{Tamojit Sadhukhan}
\address{
    Theoretical Statistics and Mathematics Unit  \\
   Indian Statistical Institute \\
    203 B. T. Road, Kolkata 700108 \\
    West Bengal, India
}
\email{tamojit96sadhukhan@gmail.com}
\thanks{The first author is partially supported by Matrics grant no MTR/2019/001448 from SERB, Govt of India. The second author is partially supported by a DST SwarnaJayanti Fellowship. The third author is partially supported by an IMU Breakout Graduate Fellowship.}

\begin{document}

\begin{abstract}
Elephant random walk is a special type of random walk that incorporates the memory of the past to determine its future steps. The probability of this walk taking a particular step ($+1$ or $-1$) at a time point, conditioned on the entire history, depends on a linear function of the proportion of steps of that type till that time point. {In this work, we consider a generalization of the elephant random walk where we investigate how the dynamics of the random walk will change if we replace this linear function with a generic map satisfying some analytic conditions. We propose a new model, called the multidimensional generalized elephant random walk, that includes several variants of elephant random walk in one and higher dimensions and generalizations thereof.} Using tools from the theory of stochastic approximation, we derive the asymptotic behavior of our model leading to newer results on the phase transition boundary between diffusive and {non-diffusive} regimes. In the process, we extend some results on one-dimensional stochastic approximation process, which can be of independent interest. We also mention a few open problems in this context.
\end{abstract}



\maketitle

\section{Introduction}\label{sec:intro}

Random walk models find many applications in physics, biology, neuroscience, computer science, econometrics, etc. Among various random walk models, the simple random walk with steps independent of each other, has been studied extensively. Exhibiting diffusive behavior, it grows at a rate equal to the square root of the number of steps taken. However, anomalous diffusion appears in many physical, biological, or social systems whose analysis often requires theoretical models that include the memory of the process. Elephant random walk, as a random walk model with memory, was introduced in \cite{erw1} to study the effect of memory on random walks. Unlike the simple random walk, it exhibits anomalous diffusion due to the incorporation of the memory of its entire history. The elephant random walk has garnered considerable attention in the last two decades; see, e.g. \cite{Bercu_2018, bertoin2022counting, 10.1063/1.4983566, Coletti_2017, Coletti_2021}. 

Like the simple random walk, the elephant random walk is a one-dimensional discrete-time random walk in which the walker moves along the integer line $\mathbb{Z}$, one step ($\pm 1$) at a time. The walker in the elephant random walk is also referred to as an elephant as it retains and uses the memory of the entire history of the walk. The walk begins from the origin. At epoch 1, the elephant moves to the right ($+1$) with probability $q$ or to the left ($-1$) with probability $1 - q$, for some $0 < q < 1$. Subsequently, for all future epochs, it chooses a step uniformly at random from the previous steps and then repeats it with probability $p$, or moves in the direction opposite to the chosen past step with probability $1-p$. Here, $p \in (0,1)$ is the important parameter that indicates the strength of the memory of the elephant and $q \in (0,1)$ is the other (not so relevant) parameter that specifies its initial distribution. An important question regarding the elephant random walk pertains to the impact of the memory parameter $p$ on the asymptotic behavior of the model. Depending on the value of $p$, it exhibits distinct behavior and undergoes a phase transition at $p = 3/4$, where it changes from diffusive $(p < 3/4)$ to superdiffusive $(p \geq 3/4)$ growth due to the effect of memory. 

{Even though the process uses the history from the entire past,} there is a time-inhomogeneous Markovian structure inherent in the dynamics of the elephant random walk. The probability that the next step of the elephant is $+1$ (respectively,~$-1$), conditioned on the past, is a linear function of the current proportion of $+1$ (respectively,~$-1$) steps. Because of the nature of its propagation, the elephant random walk is ideal for modeling scenarios in which a series of actions are taken dynamically, among two (or more) choices of actions, such that the next action depends on the relative frequencies of each of the ones taken in the past. For example, a customer wants to buy a product over other competing products based on how often the product was purchased in the past, or YouTube suggests a video to the user based on how many times the user watched similar videos in the past, etc. 

However, one limitation of using the elephant random walk to model such situations is that the relative frequencies of the actions in the past may not be observable quantities. Instead, one can readily observe the values of a possibly nonlinear function of the same. For example, how frequently a product has been purchased in the past (subsequently, we refer to the relative frequency of past purchases of a product as the \emph{market share} of that product) is usually not known to the customers, but the current price of the product, which can be viewed as a nonlinear function of the present market share, is always known and can be used to select one of the products; see Section~\ref{subsec:econ-ldrw} below. Motivated by this limitation, in Section~\ref{sec:ldrw}, we propose a generalization of the elephant random walk (called the one-dimensional generalized elephant random walk) where we generalize the probability of the next step being $+1$ (respectively, $-1$). Instead of linearly depending on the current proportion of $+1$ (respectively, $-1$) steps, the probability of the next step being $+1$ (respectively, $-1$) depends on a potentially nonlinear function of the same.

Several variations of elephant random walk have also been considered in recent times. For example, a unidirectional elephant random walk model (known as the minimal random walk) was introduced in \cite{harbola2014memory} (see also \cite{coletti2019limit}). Elephant random walk with random step sizes was analyzed in \cite{erwrs2} and \cite{erwrs1}. Elephant random walk has been extended to higher dimensions in \cite{bercu2019multi}. In Section~\ref{sec:gerw}, we consider similar generalizations of these variations of elephant random walk. We propose a very general random walk model whose various special cases provide the generalizations of the above-mentioned variations. We call this general random walk model the multidimensional generalized elephant random walk. It is needless to say that the one-dimensional generalized elephant random walk is also a special case of this model and its behavior is analyzed more intricately than its multidimensional counterpart leading to finer results in the one-dimensional case.

We establish a connection between the multidimensional generalized elephant random walk and a class of recursive algorithms, known as stochastic approximation, for determining its asymptotic behavior. Stochastic approximation was originally introduced by \cite{robbins1951stochastic} 
as a method of finding the root of an unknown function in presence of random noise. 
Since then, this tool has become quite popular with wide applications in various areas including but not limited to econometrics, clinical trials, queuing networks, wireless communications, manufacturing systems, neural nets, etc. A survey of stochastic approximation algorithms and their applications can be found in \cite{harold1997stochastic}.

In summary, our main contributions are the following. We obtain a single random walk model, namely the multidimensional generalized elephant random walk, that brings the elephant random walk and many of its one and multidimensional variants and their extensions (including the generalizations proposed by us) under the same umbrella. {Borrowing tools and techniques from the theory of stochastic approximation, we establish almost sure convergence of the scaled location of the random walk (see Theorems~\ref{thm:slln},~\ref{thm:slln-gerw}) and study the phase transition %
of the fluctuations of the scaled location around its almost sure limit (see Theorems~\ref{thm:clt},~\ref{thm:clt-gerw}).} With the help of a detailed analysis, we also derive the law of iterated logarithm in the one-dimensional case (see Theorems~\ref{thm:lil},~\ref{thm:lil-gerw-cric}), which enables us to investigate the recurrence and transience of the one-dimensional generalized elephant random walk (see Proposition~\ref{prop:tra-rec}). Further, in this case, we provide an expansion of the scaled location of the walk around its almost sure limit in the superdiffusive regime. The order of the expansion depends on the smoothness of the underlying function, giving the probability of selecting a particular type of past step (see Theorems~\ref{thm:super-dev1},~\ref{thm:super-dev1m}). {We also give detailed proofs for the higher order expansions of the one-dimensional stochastic approximation process and provide expressions for the coefficients in the expansion (see Theorems~\ref{thm:A1},~\ref{thm:A2}).} We provide several interesting and illustrative examples (see Section~\ref{subsec:examples-ldrw}), which can be of independent interest. We also propose a bunch of conjectures and potential future directions (see Open Problems~\ref{conj1},~\ref{conj2},~\ref{conj3}).

The rest of the paper is organized as follows. In Section~\ref{sec:ldrw}, we describe the one-dimensional generalized elephant random walk model along with its possible applications and present the main results for this model. We formally define the multidimensional generalized elephant random walk in Section~\ref{sec:gerw} along with important examples and statements of the main theoretical results for the same. Section~\ref {sec:sa} discusses stochastic approximation procedures and their connections to our models, as well as all the supporting results and their proofs. Finally, the proofs of the main theoretical results are given in Section~\ref{sec:proofs}. 

\section{One-dimensional Generalized Elephant Random Walk}\label{sec:ldrw}

In this section, we consider the one-dimensional generalized elephant random walk introduced in Section~\ref{sec:intro}. Section~\ref{subsec:model-ldrw} depicts the dynamics of the 
random walk. In Section~\ref{subsec:econ-ldrw}, an application motivating the model is illustrated. One-dimensional generalized elephant random walk can be connected to some other existing processes in the literature, which we describe in Section~\ref{subsec:connection}. In Section~\ref{subsec:results-ldrw}, we describe the main results for this model. We conclude Section~\ref{sec:ldrw} by providing some interesting examples in Section~\ref{subsec:examples-ldrw}. 

We begin by recalling the formal description of the elephant random walk. Fix $p,q \in (0,1)$. Let $S_n$ denote the location of the walker (i.e. the elephant) at time $n \geq 0$. Then $S_0 = 0$ and for $n\geq 0$, 
\[
    S_{n+1} = S_{n} + X_{n+1},
\]
where, for each $n \geq 0$, $X_{n+1}$ is the $(n+1)$-th step of the elephant defined as follows:
\begin{align}\label{eq00}
           X_1 = \begin{cases}
               + 1 &\text{ with probability } q, \\ 
               - 1 &\text{ with probability } 1-q,
           \end{cases}
\end{align}
and for $n \geq 1$, with $U_{n+1} \sim$ Uniform$\{1, \ldots, n\}$ independent of $\{X_1, \ldots, X_n\}$,
\begin{align}\label{eq01}
           X_{n+1} = \begin{cases}
               + X_{U_{n+1}} &\text{ with probability } p, \\ 
               - X_{U_{n+1}} &\text{ with probability } 1-p.
           \end{cases}
\end{align}

The walk has the following Markovian structure. Given $n \geq 1$, let $V_n$ (respectively, $W_n$) be the number of $+1$ (respectively, $-1$) steps till time $n$. In other words, \[{V_n} := \#\left\{1 \leq i \leq n : X_i = +1\right\}, \quad {W_n} := \#\left\{1 \leq i \leq n : X_i = -1\right\} = n - V_n.\] At time $n+1$, the probability of selecting $+1$ (respectively, $-1$) step from the previous steps is $V_n/n$ (respectively, $W_n/n$). Thus for $n \geq 1$, given $V_1, \ldots, V_n$,
\begin{align}\label{eq4}
           V_{n+1} &= \begin{cases}
              V_n+1 &\text{ with probability } p{\frac{V_n}{n}} + (1-p)\left(1-{\frac{V_n}{n}}\right),\\ 
              V_n &\text{ with probability } p{\left(1-{\frac{V_n}{n}}\right)} + (1-p){\frac{V_n}{n}},
           \end{cases}\\ \text{ and } \hspace{1cm} W_{n+1} &= n+1 - V_{n+1}. \nonumber
\end{align}
This implies that $\{V_n\}_{n\geq1}$ is a time-inhomogeneous Markov chain 
and for $n \geq 1$, \begin{align}\label{neweq:1}\mathbb{P}\left(X_{n+1} = +1 \mid X_1, \ldots, X_n\right) = p{\frac{V_n}{n}} + (1-p)\left(1-{\frac{V_n}{n}}\right).\end{align} Our model replaces the term ${V_n}/{n}$ in \eqref{neweq:1} by a function of the same.

\subsection{The Model}\label{subsec:model-ldrw}

As before, fix $p, q \in (0,1)$. The walk starts from $S_0 = 0$ at time $n = 0$. For $n\geq 0$, the location $S_{n+1}$ of the elephant at time $n+1$ is given by,
\[
    S_{n+1} = S_{n} + X_{n+1},
\]
where, %
\begin{align*}
           X_1 = \begin{cases}
               + 1 &\text{ with probability } q, \\ 
               - 1 &\text{ with probability } 1-q,
           \end{cases}
\end{align*}
and for $n \geq 1$, %
given $X_1, \ldots, X_n$,
\begin{align}\label{eq2}
           X_{n+1} = \begin{cases}
              + 1 &\text{ with probability } p{f\left(\frac{V_n}{n}\right)} + (1-p)\left(1-{f\left(\frac{V_n}{n}\right)}\right),\\ 
              - 1 &\text{ with probability } p\left(1-{f\left(\frac{V_n}{n}\right)}\right) + (1-p){f\left(\frac{V_n}{n}\right)},
           \end{cases}
\end{align}
where $f : [0,1] \to [0,1]$ and ${V_n} := \#\left\{1 \leq i \leq n : X_i = +1\right\}$ as above. We interpret \eqref{eq2} in the following way. At time $n + 1$, $n \geq 1$, the elephant chooses a step from the set of possible steps $\{-1,+1\}$ as follows. Unlike the classical elephant random walk where it chooses the step uniformly from the previous steps, the chosen step is $+1$ (respectively, $-1$) with probability $f(V_n/n)$ (respectively, $1 - f(V_n/n)$). Then, the elephant {repeats the chosen step with probability $p$, or moves in the opposite direction with probability $1-p$}. This model reduces to the elephant random walk for $f(x) = x$.  
Note that, though the probability of choosing $+1$ step at epoch $n+1$ is $f(V_n/n)$, the same for $-1$ step, in general, is not equal to $f(W_n/n)$. So, there is an asymmetry in the model dynamics unless 
\begin{equation}\label{eq:sym} f(x)+f(1-x) = 1.\end{equation} We call the one-dimensional generalized elephant random walk \emph{symmetric} if $f$ satisfies~\eqref{eq:sym}. 

\begin{remark}\label{rem:dual}
    Consider the one-dimensional generalized elephant random walk with $f$ replaced by $f^*$, given by, $f^*(x) = 1-f(x)$ and $p$ replaced by $p^* = 1-p$. Then, the dynamics of the walk remain the same. Observe that the graph of $f^*$ is the reflection of the graph of $f$ with respect to the line $y = 1/2$. Additionally, if $f$ satisfies \eqref{eq:sym}, then $f^*(x) = f(1-x)$ and the graph of $f^*$ is the reflection of the graph of $f$ with respect to the point $(1/2, 1/2)$.
\end{remark}

\subsection{An Illustrative Application}\label{subsec:econ-ldrw}

As mentioned earlier, the one-dimensional generalized elephant random walk is useful for modeling situations where actions are taken dynamically, depending on past actions. A typical example of such a situation is the following. Consider a market with two competing brands $D$ and $S$ of the same product and a customer can choose between them. We want to know the difference between the market shares of $D$ and $S$ in the long run (recall that the market share of a brand is the relative frequency of purchases of that brand in the past). {One popular model for analyzing such markets is the Theory of Increasing Returns (see \cite{ldpir} and the references therein). %
Here we present a related but different model.} Let us denote the price of $D$ (respectively, $S$) by $\pi_D$ (respectively, $\pi_S$). We assume that the price of a brand is determined by its demand or the corresponding market share. %
If $x$ denotes the market share of $D$ at a given time (observe that $0 \leq x \leq 1)$, the market share of $S$ at that time is $1-x$. Thus, at that time, the price of $D$ (respectively, $S$) is $\pi_D(x)$ (respectively, $\pi_S(1-x)$). The customers initially decide which brand to buy depending on the difference between the prices of $D$ and $S$, namely, $\pi_D(x) - \pi_S(1-x)$. The exact dependence structure varies from market to market. For illustration, we use the following. If $\pi_D(x) - \pi_S(1-x)$ is less than a threshold $L < 0$ (respectively, more than a threshold $U > 0$), the customers initially prefer $D$ (respectively, $S$), as the price of $D$ is significantly less (respectively, greater) than $S$. If $\pi_D(x) - \pi_S(1-x) \in (L, U)$, i.e., when the prices are not much different, the customers follow a randomized rule and initially prefer $D$ if and only if $\pi_D(x) - \pi_S(1-x) < X$, where, $X$ is a Uniformly distributed random variable on $(L, U)$. But, other factors (e.g. tempting advertisements or discounts) may affect this initial preference and so, ultimately, the customers buy the initially preferred brand (respectively, the other brand) with probability $p$ (respectively, with probability $1-p$). The correspondence between the customers' decisions and the steps of the random walk is straightforward. For $n \geq 0$, if $D$ (respectively, $S$) is purchased at time $n+1$, we set $X_{n+1} = 1$ (respectively, $X_{n+1} = -1$). Then, $S_n = \sum_{i=1}^n X_i$ represents the market dominance of $D$ over $S$ at time $n$. It follows that $S_n$ is a one-dimensional generalized elephant random walk with the corresponding $f$ given by,
\begin{align*}
       f(x) = \begin{cases}
       1, & \pi_D(x) - \pi_S(1-x) \leq L, \\
       \frac{U - \pi_D(x) + \pi_S(1-x)}{U - L} & \pi_D(x) - \pi_S(1-x) \in (L,U), \\
       0, & \pi_D(x) - \pi_S(1-x) \geq U.
       \end{cases}
\end{align*}
In Section~\ref{subsec:examples-ldrw} (see Example~\ref{geg-0}), we illustrate the long-run behavior of such a market. We emphasize that the above situation can not be modeled using the classical elephant random walk. {We would like to mention that the Theory of Increasing Returns can also be modeled similarly as a one-dimensional generalized elephant random walk (see, for example, \cite{franchini} and Section~\ref{sss:urn}).} %

\subsection{Connections with Other Processes}\label{subsec:connection}

The one-dimensional generalized elephant random walk can be represented in terms of various other processes existing in the literature. 

\subsubsection{Generalized Urn Process}\label{sss:urn}

A connection of the elephant random walk to the Urn Model has been established by \cite{baur2016elephant} (see also \cite{MR2413290}). A similar one-to-one correspondence exists between the one-dimensional generalized elephant random walk and the generalized urn process, which was introduced in \cite{hill1980strong}. The generalized urn process describes the discrete-time evolution of an urn containing balls of two colors, say, red and black. The urn composition at time $n \geq 1$ is denoted by $U_n = (R_n, B_n)$, where $R_n$ and $B_n$ are the numbers of red balls and black balls, respectively, at that time. We start with an empty urn. At time $n=1$, a red ball is added to the urn with probability $q$, or a black ball is added to the urn with probability $1 - q$. At time $n \geq 2$, a red ball (respectively, a black ball) is drawn from the urn with probability $f(R_n/n)$ (respectively, $1 - f(R_n/n))$. Then the drawn ball is returned to the urn along with another ball of the same color with probability $p$ or a ball of the opposite color with probability $1 - p$. The following relation describes the above-mentioned one-to-one correspondence \[(S_n : n \geq 1) \stackrel{d}{=} (R_n - B_n : n \geq 1).\]
{When the walk is symmetric (i.e. $f$ satisfies~\eqref{eq:sym}), the corresponding generalized urn process also becomes symmetric (see, for example, \cite{sym} for symmetric generalized urn process).} 
\subsubsection{Location-dependent Random Walk}\label{sss:rwre}

The dynamics of the one-dimensional generalized elephant random walk can also be described as a location-dependent random walk in the sense that the next location of the elephant depends on its current (time-normalized) location. For $n \geq 1$, as $S_n = 2V_n - n$, we can rewrite~\eqref{eq2} as: 
\begin{align}\label{eq1}
           S_{n+1} = \begin{cases}
             S_n + 1 &\text{ with probability } \frac{1}{2} + \left(p-\frac{1}{2}\right)g\left({\frac{S_{n}}{n}}\right),\\ 
              S_n - 1 &\text{ with probability } \frac{1}{2} - \left(p-\frac{1}{2}\right)g\left({\frac{S_{n}}{n}}\right),
           \end{cases}
\end{align}
where $g:[-1,1] \to [-1,1]$ is given by \begin{align}\label{eq:gf}g(x) = 2f\left(\frac{x+1}{2}\right) - 1.\end{align} This puts one-dimensional generalized elephant random walk in one-to-one correspondence with this location-dependent random walk and we shall study them interchangeably. Note that $f$ satisfies \eqref{eq:sym} if and only if $g$ is an odd function, as evident from \eqref{eq:gf}. %

\subsubsection{Sequence of Dependent Bernoulli Random Variables}\label{subsec:bernoulli}

The sequence $\left(Y_n\right)_{n\geq 1}$, given by, $Y_1 = V_1$, $Y_n = V_n - V_{n-1}$, $n \geq 2$, is a sequence of dependent Bernoulli variables. By rewriting~\eqref{eq2}, the evolution of this sequence is given by, for $n \geq 1$,
\begin{align*}\mathbb{P}\left(Y_{n+1} = +1 \mid Y_1, \ldots, Y_n \right) = (1-p) + (2p-1)f\left({\frac{1}{n}}\sum_{i=1}^n Y_i\right).\end{align*}
For a discussion on such sequences of Bernoulli variables, see \cite{wu2012asymptotics} and the references therein.

\subsection{Main Results}\label{subsec:results-ldrw}

The function $f$, and equivalently $g$, plays a crucial role in determining the dynamics of the walk. Define $h : [0,1] \mapsto [0,1]$, given by, \begin{align}\label{eq:h} h(x) := (1-p) + (2p-1)f(x) = \frac{1}{2}+\left(p-\frac{1}{2}\right)g(2x-1). \end{align} From $\eqref{eq:sym}$, it follows that the one-dimensional generalized elephant random walk is symmetric if and only if $h(x)+h(1-x) = 1$. Our first result gives a sufficient condition for the almost sure convergence of the walk. 
\begin{theorem}[Almost sure convergence]\label{thm:slln}
Suppose there exists unique $y_0 \in (0,1)$ such that $h(y_0) = y_0$ 
and for any closed $C \subseteq (0,1)\setminus \{y_0\}$, 
\begin{align}\label{eq3}
    \sup_{y \in C} \left\{(y-y_0)\left(h(y) - y\right)\right\} < 0.
\end{align}
Then, 
\begin{align}\label{eq:ldrw-conv}
\frac{S_n}{n} \stackrel{a.s.}{\to} s_0 := 2y_0 - 1.
\end{align}
\end{theorem}

\begin{remark}
Note that if $h$ (equivalently, $f$ and $g$) is continuous and 
\[
   (y-y_0)\left(h(y) - y\right) < 0 \text{ for all } y \neq y_0,
\]
then \eqref{eq3} automatically holds for all closed set $C \subseteq (0,1)\setminus \{y_0\}$. In that case, the graph of $h$ crosses the diagonal at $y_0$ from above to below. {Following \cite{hill1980strong}, $y_0$ is called a \emph{downcrossing} and \eqref{eq3} is referred to as downcrossing condition.}
\end{remark}

\begin{remark}
Suppose the walk is symmetric (i.e., $f$ satisfies \eqref{eq:sym}) so that $h(1/2) = 1/2$. If \eqref{eq3} is satisfied with $y_0 = 1/2$, then \[
   \frac{S_n}{n} \stackrel{a.s.}{\to} 0.
\]
\end{remark}

\begin{remark}
    {In case of continuous $f$, Theorem~\ref{thm:slln} follows from Theorem 6.1 of \cite{hill1980strong}. However, Theorem~\ref{thm:slln} holds for a more general class of functions and we give a different proof.}
\end{remark}

If the walk almost surely converges to some non-random $s_0 \in (-1,1)$, then the second-order behavior of the walk depends on $\eta := h'((s_0+1)/2)$ (assuming $h$ to be differentiable in a neighborhood of $(s_0+1)/2$) and different behaviors are observed for different values of $\eta$. If the assumptions of Theorem~\ref{thm:slln} holds with $s_0 = 2y_0 - 1$ and $h$ is differentiable in a a neighborhood of $y_0 = (s_0+1)/2$, then, from~\eqref{eq3}, we necessarily have
\begin{align*}
    \eta = h'(y_0) & 
    = \lim_{y\to y_0} \frac{h(y) - y}{y-y_0} + 1
    = \lim_{y\to y_0} \frac{\left(y-y_0\right)\left(h(y) - y\right)}{\left(y-y_0\right)^2} + 1
    \leq 1.
\end{align*} Similar to the elephant random walk, a phase transition happens at $\eta = 1/2$. We call the regimes diffusive, critical, or supercritical accordingly as $\eta$ $<$, $=$ or $>$ $1/2$ (see Theorem~\ref{thm:clt} below). The following result describes the rate of almost sure convergence of $S_n/n$ in the diffusive and critical regimes. For the same in the supercritical regime, see Theorem~\ref{thm:super-dev1}. 
 
\begin{theorem}[Law of iterated logarithm]\label{thm:lil}
Suppose the assumptions of Theorem~\ref{thm:slln} are satisfied so that almost surely ${S_n}/{n}$ converges to some $s_0\in (-1,1)$. Let $h$ be differentiable in a neighborhood of $(s_0+1)/2$ 
with $\eta := h'((s_0+1)/2) \leq 1/2$. Then, we have the following.
\begin{enumerate}
\item[a)] Diffusive regime: if $\eta < {1}/{2}$, almost surely
\begin{align}\label{eq:lil-ldrw1}
    &\limsup_{n \to \infty}\sqrt{\frac{n}{2\log\log n}} \left(\frac{S_n}{n} - s_0\right)  \nonumber 
    \\ = - &\liminf_{n \to \infty}\sqrt{\frac{n}{2\log\log n}} \left(\frac{S_n}{n} - s_0\right) = \sqrt{\frac{1-s^2_0}{1-2\eta}}. 
\end{align}
\item[b)] Critical regime: if $\eta = {1}/{2}$ and $h$ is further twice differentiable at $(s_0+1)/2$, then almost surely
\begin{align}\label{eq:lil-ldrw2}
    &\limsup_{n \to \infty}\left(\frac{n}{2\log n\log\log\log n}\right)^{1/2} \left(\frac{S_n}{n} - s_0\right) \nonumber \\ = - &\liminf_{n \to \infty}\left(\frac{n}{2\log n\log\log\log n}\right)^{1/2} \left(\frac{S_n}{n} - s_0\right) = \sqrt{{1-s^2_0}}. 
\end{align}
\end{enumerate}
\end{theorem} 

The following is an immediate consequence of Theorem~\ref{thm:slln} and Theorem~\ref{thm:lil}.

\begin{corollary}\label{cor:rec}
Suppose the assumptions of Theorem~\ref{thm:slln} are satisfied with $s_0 \neq 0$. Then $s_0 > 0$ (respectively, $s_0 < 0$) implies that almost surely 
\begin{align*}
\lim_{n\to\infty} S_n = \infty \text{ (respectively, }\lim_{n\to\infty} S_n = -\infty). 
\end{align*}
On the other hand, if the assumptions of Theorem~\ref{thm:lil} are satisfied with $s_0 = 0$ and $\eta \leq 1/2$, then almost surely
\[
      \left\{ \liminf_{n\to\infty} S_n = - \infty, \quad \limsup_{n\to\infty} S_n = \infty\right\}.
\]
\end{corollary}

Corollary~\ref{cor:rec} suggests a study of whether the one-dimensional generalized elephant random walk is transient or recurrent. However, as the Markovian nature of the one-dimensional generalized elephant random walk is time-inhomogeneous, the notion of recurrence and transience for this walk needs to be recalled.

\begin{definition}[Transience and Recurrence]\label{def:trrec}
    We call the one-dimensional generalized elephant random walk transient (respectively, recurrent) if the walk, starting from the origin, returns to the origin finitely many times (respectively, infinitely often) with probability one.  
\end{definition}

Because of the time-inhomogeneous Markovian structure, it is not necessary that the one-dimensional generalized elephant random walk is either transient or recurrent in the sense of Definition~\ref{def:trrec}. However, the following result follows immediately from Corollary~\ref{cor:rec}.

\begin{proposition}[Transience and Recurrence]\label{prop:tra-rec}
    Suppose the assumptions of Theorem~\ref{thm:slln} are satisfied with $s_0 \neq 0$. Then the one-dimensional generalized elephant random walk is transient in all three regimes. On the other hand, if the assumptions of Theorem~\ref{thm:lil} are satisfied with $s_0 = 0$, then the one-dimensional generalized elephant random walk is recurrent in the diffusive and critical regimes.
\end{proposition}

We have the following open problem regarding the transience/recurrence of the walk in the supercritical regime.
\begin{conjecture}\label{conj1}
    If the assumptions of Theorem~\ref{thm:slln} are satisfied with $s_0 = 0$ and $h$ is differentiable around $1/2$ with $h'(1/2) > 1/2$, then transience or recurrence of the walk is an open problem. We conjecture that the walk is transient in this case. 
\end{conjecture}

Our next result depicts the behavior of fluctuations of the walk around the almost sure limit of Theorem~\ref{thm:slln}  in the three regimes.

\begin{theorem}[Fluctuations around the almost sure limit]\label{thm:clt}
Suppose the assumptions of Theorem~\ref{thm:slln} are satisfied so that almost surely ${S_n}/{n}$ converges to some $s_0\in (-1,1)$. Let $h$ be differentiable in a neighborhood of $(s_0+1)/2$ 
with $\eta := h'((s_0+1)/2) < 1$. Then, we have the following.
\begin{enumerate}
\item[a)] Diffusive regime: if $\eta < {1}/{2}$, 
\begin{align}\label{eq:clt-ldrw1}
    \sqrt{n}\left(\frac{S_n}{n} - s_0\right) \stackrel{d}{\to} N\left(0,\frac{1-s^2_0}{1-2\eta}\right).
\end{align}
\end{enumerate}

For $\eta \geq 1/2$, we further assume that $h$ is twice differentiable at $(s_0+1)/2$.

\begin{enumerate}
\item[b)] Critical regime: if $\eta = {1}/{2}$, 
\begin{align}\label{eq:clt-ldrw2}
    \sqrt{\frac{n}{\log n}}\left(\frac{S_n}{n} - s_0\right) \stackrel{d}{\to} N\left(0,1-s^2_0\right).
\end{align}
\item[c)] Supercritical regime: if $1/2 < \eta < 1$, there exists a finite random variable $L$ (that may depend on $h$ and $q$) such that
\begin{align}\label{eq:clt-ldrw3}
    n^{1-\eta}\left(\frac{S_n}{n} - s_0\right) \stackrel{a.s.}{\to} L.
\end{align}
\end{enumerate}
\end{theorem}

The above result naturally gives rise to the following open problems. 

\begin{problem}\label{conj2}
    For $1/2 < \eta < 1$, almost nothing is known about the properties of the distribution of the limiting random variable $L$ obtained in~\eqref{eq:clt-ldrw3}.
\end{problem}
\begin{problem}\label{conj3}
    Theorem~\ref{thm:clt} does not consider the case $\eta = 1$. In fact, for $\eta = 1$, the behavior of fluctuations of the walk around the almost sure limit of Theorem~\ref{thm:slln} is open.
\end{problem}

The next result deals with the behavior of the walk in the supercritical regime in a more elaborate manner. For a simpler version of this result, see Corollary~\ref{cor:supdiff}.

{
\begin{theorem}\label{thm:super-dev1}
Suppose the assumptions of Theorem~\ref{thm:slln} are satisfied so that almost surely ${S_n}/{n}$ converges to some $s_0\in (-1,1)$. Let $h$ be differentiable in a neighborhood of $(s_0+1)/2$ 
with $\eta := h'((s_0+1)/2) \in  (1/2,1)$ and $L$ be as in Theorem~\ref{thm:clt}. 
If $h$ has further $m$ derivatives, i.e. $(m+1)$ derivatives in all, around $(s_0+1)/2$, then there exist constants $\beta_1 = 1$, $\beta_2, \ldots, \beta_{m+1}$ %
recursively given by \begin{align}\label{eq:beta_ldrw}\beta_{j+1} &= -\frac{1}{j(1-\eta)}\sum_{i=2}^{j+1}\frac{h^{(i)}\left(\frac{s_0+1}{2}\right)}{2^{i-1}i!}\sum_{(c_1, \ldots, c_i) \in \mathcal{P}_{i,j+1}}\nu_{(c_1, \ldots, c_i)}\beta_{c_1}\ldots \beta_{c_i}, \quad j = 1, \ldots, m,\end{align} 
where, for $1 \leq i \leq t$, 
\begin{align}\label{eq:combn}
\mathcal{P}_{i,t} &= \left\{(c_1, \ldots, c_i) : \sum_{i=1}^i c_i = t, 1 \leq c_1 \leq \ldots \leq c_i\right\}, \nonumber\\ 
\nu_{(c_1, \ldots, c_i)} &= \# \text{ distinct arrangements of } (c_1, \ldots, c_i),
\end{align}
such that the following holds:
\begin{enumerate}
\item[a)] If \[ m \geq \frac{\eta - 1/2}{1-\eta}, \text{ then for } m_0 = \left\lfloor\frac{\eta - 1/2}{1-\eta}\right\rfloor,\]
we have, almost surely,
\begin{align}\label{eq:slil-ldrw}
    &\limsup_{n\to\infty}\sqrt{\frac{n}{2\log \log n}}\left(\left(\frac{S_n}{n} - s_0\right) - \sum_{j = 0}^{m_0} \beta_{j+1}\left(\frac{L}{n^{1-\eta}}\right)^{j+1} 
    \right) \nonumber \\
    =-&\liminf_{n\to\infty}\sqrt{\frac{n}{2\log \log n}}\left(\left(\frac{S_n}{n} - s_0\right) - \sum_{j = 0}^{m_0} \beta_{j+1}\left(\frac{L}{n^{1-\eta}}\right)^{j+1} \right)
   = \sqrt{\frac{1 - s^2_0}{2\eta -1}},
\end{align}
and
\begin{align}\label{eq:sclt-ldrw}
    \sqrt{n}\left(\left(\frac{S_n}{n} - s_0\right) - \sum_{j = 0}^{m_0} \beta_{j+1}\left(\frac{L}{n^{1-\eta}}\right)^{j+1}\right) \stackrel{d}{\to} {N}\left(0, \frac{1 - s^2_0}{2\eta -1}\right).
\end{align}
\item[b)] If \[ m < \frac{\eta - 1/2}{1-\eta},\] 
then, we have, almost surely,
\begin{align}\label{eq:s-ldrw}
    \left(\frac{S_n}{n} - s_0\right) - \sum_{j = 0}^{m} \beta_{j+1}\left(\frac{L}{n^{1-\eta}}\right)^{j+1} 
    {=}  o\left({n}^{-(1-\eta)(m+1)}\right).  
\end{align}
\end{enumerate}
\end{theorem}}

{
\begin{remark}\label{rem:0}
    Theorem~\ref{thm:super-dev1} gives an expansion of $S_n/n$ around $s_0$ in terms of powers of $n^{-(1-\eta)}L$. The number of terms in the expansion depends on the smoothness of the function $h$ around $(s_0+1)/2$. If there are sufficiently large number of derivatives of $h$, then we obtain the law of iterated logarithm and asymptotic normality of the error term. If the function $h$ lacks enough number of derivatives, we only have almost sure order bound for the error term.
\end{remark}}

{
\begin{remark}\label{rem:1}
   From \eqref{eq:beta_ldrw}, it follows that if $h^{(i+1)}((s_0+1)/2) = 0$ for $1 \leq i \leq m$, then $\beta_i = 0$ for $1 \leq i \leq m$. In the classical elephant random walk, $h(x) = (2p-1)x+1-p$ %
    is infinitely differentiable around $(s_0+1)/2$ (with $s_0 = 0$), and $\beta_i = 0$ for all $i \geq 1$. Further assume $1/2 < \eta  = 2p-1 < 1$ (or equivalently, $3/4 < p < 1$). Then Theorem~\ref{thm:super-dev1}a) applies with the sum in the expansion having only the linear term, which is $n^{-(1-\eta)}L = n^{-2(1-p)}L$.   
    This agrees with Theorem 2.3 of \cite{kubota2019gaussian}.
\end{remark}}

The following result is a special case of Theorem~\ref{thm:super-dev1} when $m = 1$.

\begin{corollary}\label{cor:supdiff}
   Suppose the assumptions of Theorem~\ref{thm:slln} are satisfied so that almost surely ${S_n}/{n}$ converges to some $s_0\in (-1,1)$. Let $h$ be twice differentiable in a neighborhood of $(s_0+1)/2$ 
with $\eta := h'((s_0+1)/2) \in  (1/2,1)$, $\eta_1 := h''((s_0+1)/2)$ and $L$ be as in Theorem~\ref{thm:clt}. 
Then the following results hold:
\begin{enumerate}
\item[a)] If ${1}/{2} < \eta < {3}/{4}$,
then almost surely,
\begin{align}\label{newq-1}
    \limsup_{n\to\infty}\sqrt{\frac{n}{2\log \log n}}&\left(\left(\frac{S_n}{n} - s_0\right) -  {n^{-(1-\eta)}}{L} 
    \right) \nonumber \\
    =-\liminf_{n\to\infty}\sqrt{\frac{n}{2\log \log n}}&\left(\left(\frac{S_n}{n} - s_0\right) - {n^{-(1-\eta)}}{L} \right)
   = \sqrt{\frac{1 - s^2_0}{2\eta -1}},\hspace{2.1cm}&
\end{align}
and
\begin{align}\label{newq-2}
    \sqrt{n}\left(\left(\frac{S_n}{n} - s_0\right) - {n^{-(1-\eta)}}{L}\right) \stackrel{d}{\to} {N}\left(0, \frac{1 - s^2_0}{2\eta -1}\right).
\end{align}
\item[b)] If $\eta = {3}/{4}$,
then almost surely,
\begin{align}\label{newq-3}
    \limsup_{n\to\infty}\sqrt{\frac{n}{2\log \log n}}&\left(\left(\frac{S_n}{n} - s_0\right) - {n^{-1/4}}{L} + {\eta_{1}}n^{-1/2}L^{2} 
    \right) \nonumber \\
    =-\liminf_{n\to\infty}\sqrt{\frac{n}{2\log \log n}}&\left(\left(\frac{S_n}{n} - s_0\right) - {n^{-1/4}}{L} + {\eta_{1}}n^{-1/2}L^{2} \right)
    = \sqrt{2\left(1- s^2_0\right)},\hspace{0.75cm}&
\end{align}
and 
\begin{align}\label{newq-4}
    \sqrt{n}\left(\frac{S_n}{n} - s_0\right) - {n^{1/4}}{L} + {{\eta_{1}}L^{2}}
    \stackrel{d}{\to} {N}\left(0, 2\left(1- s^2_0\right)\right).
\end{align}
\item[c)] If ${3}/{4} < \eta < 1$,
then 
\begin{align}\label{newq-5}
    {n}^{2(1-\eta)}\left(\left(\frac{S_n}{n} - s_0\right) - {n^{-(1-\eta)}}{L} \right)
    \stackrel{a.s.}{\to}  -\frac{\eta_{1}}{4(1-\eta)}L^{2}. 
\end{align}
\end{enumerate} 
\end{corollary}

\subsection{Examples}\label{subsec:examples-ldrw}

We now give some examples of functions that satisfy the assumptions of Theorem~\ref{thm:slln} so that the corresponding one-dimensional generalized elephant random walks converge almost surely. Consequently, we also discuss whether they undergo a phase transition. 

\begin{example}\label{geg-(-2)}
    Let $f$ be given by $f(x) = ax+b$, $a, b \in \mathbb{R}$, $0 \leq b, a+b \leq 1$. For $a = 0, b = 0.5$ \textcolor{black}{(i.e. $f(x) = 0.5$)}, this reduces to the simple symmetric random walk, and for $a = 1, b = 0$ \textcolor{black}{(i.e. $f(x) = x$)}, this reduces to the classical elephant random walk. If $a+2b = 1$, the walk is symmetric \textcolor{black}{(i.e. $f$ satisfies \eqref{eq:sym})}. For this choice of the function $f$, Theorem~\ref{thm:slln} implies that
    \[
        \frac{S_n}{n} \stackrel{a.s.}{\to} \frac{(2p-1)(a+2b-1)}{1-(2p-1)a}.
    \]
    Theorem~\ref{thm:clt} implies that the second-order dynamics of the walk depend on the value of $(2p-1)a$ $(< 1)$. The walk is diffusive (respectively, critical, supercritical) if and only if $(2p-1)a < 1/2$ (respectively, $(2p-1)a = 1/2$, $(2p-1)a > 1/2$).
\end{example}

\begin{example}\label{geg-(-1)}
    This example studies a quadratic function $f$, given by, 
    \[
        f(x) = \begin{cases}
            x^2+\frac{1}{4}, & 0 \leq x \leq \frac{1}{2} \\
            \frac{3}{4}-(1-x)^2, & \frac{1}{2} \leq x \leq 1.
        \end{cases}
    \]
    The walk is symmetric \textcolor{black}{(i.e. $f$ satisfies \eqref{eq:sym})}. Theorem~\ref{thm:slln} implies that, ${S_n}/{n} \stackrel{a.s.}{\to} 0$ and 
    Theorem~\ref{thm:clt} implies that the walk is diffusive (respectively, critical, supercritical) if and only if $p < 3/4$ (respectively, $p = 3/4$, $p > 3/4$). 
\end{example}

The following example illustrates the application of the one-dimensional generalized elephant random walk in modeling the market dynamics, discussed in Section~\ref{subsec:econ-ldrw}. 

\begin{example}\label{geg-0}
For this illustration, using the notations of Section~\ref{subsec:econ-ldrw}, we take $U = -L = 0.5$ and $\pi_D = \pi_S = \pi$, where, $\pi: [0,1] \mapsto [0,1]$ is given by, $\pi(x) = x^3/2$. The corresponding walk is symmetric (i.e., $f$ satisfies \eqref{eq:sym}). Theorem~\ref{thm:slln} implies that $S_n/n \stackrel{a.s.}{\to} 0$ and Theorem~\ref{thm:clt} implies that the walk is diffusive (respectively, critical, supercritical) if and only if $p > 1/6$ (respectively, $p = 1/6$, $p < 1/6$). %
\end{example}

{Examples~\ref{geg-2a} and \ref{geg-2}} are easier to describe in terms of the function $g$ of the equivalent location-dependent random walk (discussed in Section~\ref{sss:rwre}) rather than in terms of $f$. We can easily get the corresponding $f$ using \eqref{eq:gf}. We also recall that, as argued in Section~\ref{sss:rwre}, the walk is symmetric if and only if $g$ is odd. 

\begin{example}\label{geg-2a}
     Let $g$ be a polynomial of the form \[{g(x) = \sum_{i=1}^da_ix^i},\quad \sum_{i=1}^d i|a_i| < 1.\] The restrictions on the coefficients $\{a_i\}_{i=0}^d$ of the polynomial $g$ are sufficient conditions to ensure that $g$ takes values in $[-1,1]$ and the corresponding $f$, given by \eqref{eq:gf}, satisfies the assumptions of Theorem~\ref{thm:slln} with $s_0 = 0$. Thus, for such class of polynomials, by Theorem~\ref{thm:slln}, we get $S_n/n \stackrel{a.s.}{\to} 0$. An application of Theorem~\ref{thm:clt} indicates that the dynamics of the walk depend on the value of $(2p-1)a_1$. The walk is diffusive (respectively, critical, supercritical) if and only if $(2p-1)a_1 < 1/2$ (respectively, $(2p-1)a_1 = 3/4$, $(2p-1)a_1 > 3/4$).
\end{example}

\begin{example}\label{geg-2}
   Let $g$ be given by $g(x) = \phi(x)^k$, where, $k \in \mathbb{N}$ and $\phi$ is any odd function on $[-1,1]$ satisfying $\phi(x) < x$ for all $x \in (0,1]$. Certain choices for the function $\phi(x)$ can be  $\phi(x) = \sin x$, $\phi(x) = \tanh x$ etc. For all odd $k$, $g$ is odd and so the walk is symmetric. For each $k \in \mathbb{N}$, Theorem~\ref{thm:slln} implies that $S_n/n \stackrel{a.s.}{\to} 0$. Using Theorem~\ref{thm:clt}, we get that if $k > 1$, the walk is always diffusive. If $k=1$, the walk is diffusive (respectively, critical, supercritical) if and only if $(2p-1)\phi'(0) < 1/2$ (respectively, $(2p-1)\phi'(0) = 1/2$, $(2p-1)\phi'(0) > 1/2$).
\end{example}

{The final example provides an illustration for Corollary~\ref{cor:supdiff}}.

\begin{example}%
   {Let $f$ be given by
   \[
       f(x) = \frac{1}{2}+3\left(x-\frac{1}{2}\right) +\left(x-\frac{1}{2}\right)^{2}+\operatorname{sgn}\left(x-\frac{1}{2}\right)\left(x-\frac{1}{2}\right)^3,
   \]
   and $11/30 < p < 19 /30$. (The range of $p$ is chosen so as to make the function $h$ take values in $(0,1)$.) Theorem~\ref{thm:slln} implies that, ${S_n}/{n} \stackrel{a.s.}{\to} 0$ and 
    Theorem~\ref{thm:clt} implies that the walk is diffusive (respectively, critical, supercritical) if and only if $p < 7/12$ (respectively, $p = 7/12$, $p > 7/12$). Further, for $p > 7/12$, from Corollary~\ref{cor:supdiff} we get, if $p < {5}/{8}$ (respectively, $p = {5}/{8}$, $p > {5}/{8}$),  \eqref{newq-1} and \eqref{newq-2} (respectively, \eqref{newq-3} and \eqref{newq-4}, \eqref{newq-5}) hold with $s_0 = 0$, $\eta = 3(2p-1)$ and $\eta_1 = 2(2p-1)$.
   }
\end{example}

\section{Multidimensional Generalized Elephant Random Walk}\label{sec:gerw}

In this section, we extend the generalized elephant random walk to higher dimensions. In Section~\ref{subsec:model-gerw}, we describe the dynamics of the multidimensional generalized elephant random walk model. Section~\ref{subsec:special-cases} describes the dynamics of different variations of the elephant random walk that we consider here along with their generalizations and mentions how all these models become special cases of the multidimensional generalized elephant random walk. Finally in Section~\ref{sec:main-results}, we state the main results for the multidimensional generalized elephant random walk model. 

We use the following notations in the rest of the article. For $s \in \mathbb{N}$, denote $[s] := \{1, \ldots, s\}$. For any $s\times 1$ vector (both deterministic and random) $\boldsymbol{V}$ and $\emptyset \neq E = \{j_1 < j_2 < \ldots < j_l\} \subseteq [s]$, $\boldsymbol{V}^{(E)}$ is defined to be a $s\times 1$ vector, given by the following: 
\begin{align}\label{eq:not1}
    \left(\boldsymbol{V}^{(E)}\right)_i = \begin{cases}
        \boldsymbol{V}_i, &\quad i \in E, \\
        0, &\quad i \in [s]\setminus E,
    \end{cases}
\end{align}
and $\boldsymbol{V}_{(E)}$ is defined to be a $l\times 1$ vector, given by the following: 
\begin{align*}
    \left(\boldsymbol{V}_{(E)}\right)_{\alpha} = 
        \boldsymbol{V}_{j_{\alpha}}, &\quad \alpha \in [l].
\end{align*}
Note that $\boldsymbol{V}^{(E)}$ is obtained from $\boldsymbol{V}$ by replacing the coordinates other than those indexed by $E$ with zeros, while for $\boldsymbol{V}_{(E)}$ the coordinates other than those indexed by $E$ are dropped. We define $\boldsymbol{V}^{(\emptyset)}$ to be the zero vector but do not define $\boldsymbol{V}_{(E)}$ when $E=\emptyset$.

Similarly, for any $s\times s$ matrix $\boldsymbol{M}$ and $\emptyset \neq E = \{j_1 < j_2 < \ldots < j_l\} \subseteq [s]$, $\boldsymbol{M}^{(E)}$ is defined to be a $s\times s$ matrix, given by the following: 
\begin{align*}
    \left(\boldsymbol{M}^{(E)}\right)_{i,k} = \begin{cases}
        \boldsymbol{M}_{i,k}, &\quad i,k \in E, \\
        0, &\text{ otherwise },
    \end{cases}
\end{align*}
and $\boldsymbol{M}_{(E)}$ is defined to be a $l\times l$ matrix, given by the following: 
\begin{align*}
    \left(\boldsymbol{M}_{(E)}\right)_{\alpha,\beta} = 
        \boldsymbol{M}_{j_{\alpha},j_{\beta}}, &\quad \alpha,\beta \in [l].
\end{align*}
Note that $\boldsymbol{M}_{(E)}$ is the principal submatrix of $\boldsymbol{M}$ corresponding to the coordinates of $E$, while $\boldsymbol{M}^{(E)}$ is obtained by replacing all other entries with zeros. We define $\boldsymbol{M}^{(\emptyset)}$ to be the zero matrix but do not define $\boldsymbol{M}_{(E)}$ when $E=\emptyset$.

The vectors in this article are column vectors. A row vector is written as the transpose of the corresponding column vector. Let $\boldsymbol{0}_s$ (respectively, $\boldsymbol{1}_s$) denotes the $s$-dimensional vector with $0$ (respectively, $1$) everywhere. Let $\mathbb{I}_s$ be the $s\times s$ identity matrix. Let $\iota = \sqrt{-1}$ denotes the imaginary unit. By $\delta_{\boldsymbol{V}}$, we denote the distribution of the random variable degenerate at the deterministic vector $\boldsymbol{V}$. We denote the trace and transpose of the matrix $\boldsymbol{M}$ by $\operatorname{tr} \boldsymbol{M}$ and $\boldsymbol{M}^{\top}$, respectively. Whenever the dimensions of a vector $\boldsymbol{V}$ (respectively, matrix $\boldsymbol{M}$) is $1$ (respectively, $1 \times 1$), we denote it by $V$ (respectively, $M$). The Euclidean norm is denoted by $\|\cdot\|$.

\subsection{The Model}\label{subsec:model-gerw}

The multidimensional generalized elephant random walk $\left(\boldsymbol{S}_n\right)_{n\geq0}$ is a random walk on $\mathbb{R}^d$, $d \in \mathbb{N}$. For each $n \geq 0$, $\boldsymbol{S}_n$, the location of the walk at time $n$, is given by an affine transformation of $\boldsymbol{\widetilde{S}}_n$, the location at time $n$ of another auxiliary random walk $(\boldsymbol{\widetilde{S}}_n)_{n\geq0}$ on $[0,\infty)^{s}$, $s \in \mathbb{N}$, through 
\begin{align}\label{eq22}
    \boldsymbol{S}_n = \boldsymbol{A}\boldsymbol{\widetilde{S}}_n + n\boldsymbol{b},
\end{align}
where, $\boldsymbol{A}$ is a deterministic $d \times s$ matrix and $\boldsymbol{b}$ is a deterministic $d\times 1$ vector. The matrix $\boldsymbol{A}$ transforms the location $\boldsymbol{\widetilde{S}}_n$ of the auxiliary walk into a location on $\mathbb{R}^d$ and $\boldsymbol{b}$ is the drift component. %

To describe the dynamics of $(\boldsymbol{\widetilde{S}}_n)_{n\geq0}$, let $0 = j_0 < j_1 < j_2 < \ldots < j_{r-2} < j_{r-1} \leq j_r = s$ for $r \in [s+1]$ and define $\pi_i = \{j_{i-1}+1, \ldots, j_i\}$ for $i \in [r]$. If $j_{r-1} = j_r$, we take $\pi_r = \emptyset$. Define  $\boldsymbol{\Pi}^s_r = \{\pi_i\}_{i = 1}^r$ which forms a partition of $[s]$ of size $r$. %
The auxiliary walk $(\boldsymbol{\widetilde{S}}_n)_{n\geq0}$ starts from the origin $\boldsymbol{\widetilde{S}}_0 = \boldsymbol{0}_s$ at time $n = 0$. For $n \geq 0$,
\[
    \boldsymbol{\widetilde{S}}_{n+1} = \boldsymbol{\widetilde{S}}_{n} + \boldsymbol{\widetilde{X}}_{n+1},
\]
where, $\boldsymbol{\widetilde{X}}_{1}$ is a $D_s$-valued random variable where $D_s \subseteq [0,\infty)^s$ is an $s$-dimensional rectangle (possibly unbounded) with the origin as one of the corners.
For $n \geq 1$, given $\boldsymbol{\widetilde{X}}_1, \ldots, \boldsymbol{\widetilde{X}}_{n}$, 
\begin{align}\label{eq12}
           \boldsymbol{\widetilde{X}}_{n+1} = \begin{cases}
               \boldsymbol{Y}^{({\pi}_1)}_{n+1} &\text{ with probability }   \mathcal{P}_1\left(\frac{\boldsymbol{\widetilde{S}}_n}{n}\right), \\ \boldsymbol{Y}^{({\pi}_2)}_{n+1} &\text{ with probability }   \mathcal{P}_2\left(\frac{\boldsymbol{\widetilde{S}}_n}{n}\right), \\ \vdots \\ \boldsymbol{Y}^{({\pi}_{r-1})}_{n+1} &\text{ with probability }   \mathcal{P}_{r-1}\left(\frac{\boldsymbol{\widetilde{S}}_n}{n}\right), \\ 
               \boldsymbol{Y}^{({\pi}_r)}_{n+1} &\text{ with probability } 1-\sum_{j=1}^{r-1}\mathcal{P}_j\left(\frac{\boldsymbol{\widetilde{S}}_n}{n}\right).
           \end{cases}
\end{align}
where, $\boldsymbol{Y}_1, \boldsymbol{Y}_2, \boldsymbol{Y}_3, \ldots$ are independent and identically distributed $D_s$-valued random variables, for $n \geq 1$ and $i \in [r]$, $\boldsymbol{Y}^{(\pi_i)}_{n+1}$ is defined according to~\eqref{eq:not1} and $\mathcal{P}_i: D_s \mapsto [0,1]$, $i \in [r-1]$, satisfying $\sum_{j=1}^{r-1}\mathcal{P}_j(\boldsymbol{x}) < 1$ for all $\boldsymbol{x} \in D_s$. In other words, given $\boldsymbol{\widetilde{X}}_1, \ldots, \boldsymbol{\widetilde{X}}_{n}$, for $i \in [r-1]$, $\boldsymbol{\widetilde{X}}_{n+1}$ is the vector given by $i$-th block of $\boldsymbol{Y}_{n+1}$ appended with zero with probability $\mathcal{P}_i(\boldsymbol{\widetilde{S}}_n/n)$ or the vector given by $r$-th block of $\boldsymbol{Y}_{n+1}$ appended with zero with complementary probability.

\subsection{Some Special Cases}\label{subsec:special-cases}

In this section, we describe how some variations of the
elephant random walk can be generalized in the same way we generalized the classical elephant random walk in Section~\ref{sec:ldrw} and then show that all these models are special cases of the multidimensional generalized elephant random walk.

\subsubsection{One-dimensional Generalized Elephant Random Walk}\label{eg:ldrw}

We first write down the usual one-dimensional generalized elephant random walk in the form described in Section~\ref{subsec:model-gerw}.
\hfill \\\\
\textbf{The Model:} The evolution of this model has already been described in Section~\ref{subsec:model-ldrw}. \\\\
By taking $\widetilde{S}_n = V_n = \#\{1 \leq i \leq n : X_i = 1\}$, one can easily show that it is a multidimensional generalized elephant random walk model with the parameters $s = d = 1$, $r = 2$, ${Y}_1 \sim \delta_{{1}}$, $\boldsymbol{\Pi}^1_2 = \{\{1\},\emptyset\}$, ${A} = 2$, ${b} = -1$ and ${\mathcal{P}_1} = h$, where $h$ is given by \eqref{eq:h}.

\subsubsection{Generalized Minimal Random Walk}\label{eg:gmrw}

This is a generalization of the minimal random walk model, first introduced in \cite{harbola2014memory}. 

\noindent\textbf{The Model:} The walk starts from the origin at time $0$. At time $n = 1$, the elephant moves one step in the positive direction (respectively, does not move) with probability $r \in (0,1)$ (respectively, $1 - r$). For $n \geq 1$, let $V_n$ be the number of $+1$ steps till time $n$ (observe that, here $V_n = S_n$). Let $f : [0,1] \to [0,1]$. At time $n + 1$, $n \geq 1$, the elephant chooses a step $\mathcal{X}_{n+1}$. 
Given $S_n$,  $\mathcal{X}_{n+1} = +1$ (respectively, $\mathcal{X}_{n+1} = 0$) with probability $f(V_n/n)$ (respectively, $1-f(V_n/n)$). If $\mathcal{X}_{n+1} = +1$, the elephant takes the step $+1$ (respectively, $0$) with probability $p \in (0,1)$ (respectively, $1-p$). If $\mathcal{X}_{n+1} = 0$, the elephant takes the step $+1$ (respectively, $0$) with probability $q \in (0,1)$ (respectively, $1-q$). Thus for any $n\geq 0$, the location $S_n$ of the elephant at time $n$  is given by 
\[
    S_{n+1} = S_{n} + X_{n+1},
\]
where $S_0 = 0$, 
\begin{align*}
   X_{1} = \begin{cases} +1 &\text{ with probability } r, \\ 0
   &\text{ with probability } 1-r, \end{cases} \quad 0 < r < 1.
\end{align*}
and $X_{n+1}$, $n \geq 1$, are as follows. For $n \geq 1$, 
\begin{align*}
    \mathcal{X}_{n+1} = \begin{cases} +1 &\text{ with probability } f\left(\frac{S_n}{n}\right), \\ 0
   &\text{ with probability } 1-f\left(\frac{S_n}{n}\right).\end{cases}
\end{align*}
Given $\mathcal{X}_{n+1}$, if $\mathcal{X}_{n+1} = +1$,
\begin{align*}
   X_{n+1} = \begin{cases} +1 &\text{ with probability } p, \\ 0
   &\text{ with probability } 1-p, \end{cases} \quad 0 < p < 1,
\end{align*}
and if $\mathcal{X}_{n+1} = 0$,
\begin{align*}
   X_{n+1} = \begin{cases} +1 &\text{ with probability } q, \\ 0
   &\text{ with probability } 1-q, \end{cases} \quad 0 < q < 1.
\end{align*}
This model reduces to the minimal random walk of \cite{harbola2014memory} for $f(x) = x$. Also, for $p=q$, the steps $X_n$ follow Bernoulli($p$) distribution, whenever $n>1$, and the model reduces to a linear transform of (possibly biased) simply random walk. Again, for $p=1-q$, we obtain the model in Section~\ref{sec:ldrw} as a special case.

By taking $\widetilde{S}_n = V_n$, we can show that it is a multidimensional generalized elephant random walk model with the parameters $s = d = 1$, $r = 2$, ${Y}_1 \sim \delta_{{1}}$, $\boldsymbol{\Pi}^1_2 = \{\{1\}, \{\emptyset\}\}$, ${A} = 1$, ${b} = 0$ and ${\mathcal{P}_1} : [0,1] \to [0,1]$, given by, $\mathcal{P}_1(x) = (p-q)f(x) + q$.

\begin{remark}
Like the one-dimensional generalized elephant random walk, there is an asymmetry in the dynamics of the generalized minimal random walk. Namely, though the probability of choosing the $+1$ step at epoch $n + 1$ is equal to $f(V_n/n)$, the same for choosing the $0$ step at epoch $n + 1$, in general, is not equal to $f(W_n/n)$, where, $W_n = n - V_n = $ is the number of $0$ steps till time $n$. We call the generalized minimal random walk symmetric if $f$ satisfies \eqref{eq:sym}. Examples of such $f$ can be found in Section~\ref{subsec:examples-ldrw}.   
\end{remark}

\subsubsection{Generalized Elephant Random Walk with Random Step Sizes}\label{eg:gerwrs}
This is a generalization of the elephant random walk with random step sizes, first introduced in \cite{erwrs1}. \hfill \\\\
\textbf{The Model:} This walk replaces the simple ($\pm 1$) steps of Model~\ref{eg:ldrw} by steps of random magnitude, while the directions evolve as earlier. The walk starts from the origin at time $0$. Let $Z_1, Z_2, \ldots$ be independent and identically distributed positive real-valued random variables with finite mean. At time $n = 1$, the elephant moves in the positive direction (respectively, negative direction) with probability $q \in (0,1)$ (respectively, $1 - q$), with the magnitude of the step being $Z_1$, independent of the direction. For $n \geq 1$, let $V_n$ be the number of steps in the positive direction till time $n$. Let $f : [0,1] \to [0,1]$. At time $n + 1$, $n \geq 1$, the elephant chooses a direction $\mathcal{X}_{n+1}$. 
Given $V_n$, $\mathcal{X}_{n+1} = +1$ (respectively, $\mathcal{X}_{n+1} = -1$) with probability $f(V_n/n)$ (respectively, $1 - f(V_n/n)$). Then the elephant moves in the direction $\mathcal{X}_{n+1}$ (respectively, $-\mathcal{X}_{n+1}$) with probability $p \in (0,1)$ (respectively, $1-p$), with the magnitude of the step being $Z_{n+1}$, independent of the direction. So, for any $n\geq 0$, the location $S_n$ of the elephant at time $n$ is given by
\[
    S_{n+1} = S_{n} + X_{n+1},
\]
where $S_0 = 0$, given $Z_1$,
\begin{align*}
   X_{1} = \begin{cases} +Z_1 &\text{ with probability } q, \\ -Z_1
   &\text{ with probability } 1 - q, \end{cases} \quad 0 < q < 1.
\end{align*} 
and $X_{n+1}$, $n \geq 1$, are as follows. For $n \geq 1$, $V_{n} = \sum_{i=1}^n \boldsymbol{1}\{X_i > 0\}$,
\begin{align*}
    \mathcal{X}_{n+1} = \begin{cases}+1 &\text{ with probability } f\left(\frac{V_n}{n}\right), \\ -1
   &\text{ with probability } 1-f\left(\frac{V_n}{n}\right).\end{cases}
\end{align*}
Given $\mathcal{X}_{n+1}$ and $Z_{n+1}$, 
\begin{align*}
   X_{n+1} = \begin{cases} +Z_{n+1}\mathcal{X}_{n+1} &\text{ with probability } p, \\ -Z_{n+1}\mathcal{X}_{n-1}
   &\text{ with probability } 1-p, \end{cases} \quad 0 < p < 1.
\end{align*}
This model reduces to the elephant random walk with random step sizes of \cite{erwrs1} for $f(x) = x$. \\\\
By taking \[\boldsymbol{\widetilde{S}_n} = \begin{pmatrix}
    V_n \\ \sum_{i=1}^nZ_i\boldsymbol{1}\{V_n - V_{n-1} = 1\} \\ \sum_{i=1}^nZ_i\boldsymbol{1}\{V_n - V_{n-1} = 0\} 
\end{pmatrix},\] it can be shown that this is a multidimensional generalized elephant random walk model with the parameters $s = 3$, $d = 1$, $r = 2$, $\boldsymbol{Y}_n \stackrel{d}{=} \begin{pmatrix}1 & Z_n & Z_n\end{pmatrix}^{\top}$, $n\geq 1$, $\boldsymbol{\Pi}^3_2 = \{\{1,2\}, \{3\}\}$, $\boldsymbol{A}_{1\times 3} = \begin{bmatrix} 0 & 1 & -1\end{bmatrix}$, $\boldsymbol{b}_{1\times 1} = 0$ and, ${\mathcal{P}}_1 : [0,1]\times[0,\infty)^2 \to [0,1]$ given by $\mathcal{P}_1(\boldsymbol{x}) = (2p-1)f(x_1) + 1-p$.

\begin{remark}
Like the one-dimensional generalized elephant random walk, there is an asymmetry in the dynamics of the generalized elephant random walk with random step sizes. Namely, though the probability of choosing the $+1$ direction at epoch $n + 1$ is equal to $f(V_n/n)$, the same for choosing the $-1$ direction at epoch $n + 1$, in general, is not equal to $f(W_n/n)$, where, $W_n = n - V_n = $ is the number of steps in the negative direction till time $n$. We call the generalized elephant random walk with random step sizes symmetric if $f$ satisfies \eqref{eq:sym}. Examples of such $f$ can be found in Section~\ref{subsec:examples-ldrw}.     
\end{remark}

\subsubsection{$k$-dimensional Generalized Elephant Random Walk}\label{eg:gmerw}

This is a generalization of $k$-dimensional elephant random walk, first introduced in \cite{bercu2019multi}. \hfill \\\\
\textbf{The Model:} The $k$-dimensional walk starts from the origin at time $0$. {For $j \in [2k]$, by the $j$-th direction, we mean the direction of $\boldsymbol{u}_{j}$}, where, \textcolor{black}{\[\boldsymbol{u}_{j} = (-1)^{j+1}(
0  \ldots  0  \underbrace{1}_{\left\lfloor\frac{j+1}{2}\right\rfloor-\text{th position}}  0 \ldots  0
)_{k \times 1},\]}that is, they denote the unit vector successively in $2k$ possible directions. At time $n = 1$, the elephant moves one step in one of $2k$ possible directions with equal probability. For $n \geq 1$, let $V^j_n$, $1 \leq j \leq 2k-1$ be the number of steps in the $j$-th direction till time $n$. Let $f : [0,1] \to [0,1]$. At time $n + 1$, $n \geq 1$, the elephant chooses a direction $\mathcal{X}_{n+1}$. 
Given $\left(V^j_n\right)_{j=1}^{2k-1}$, \begin{align}\label{eq:chi_n}\mathcal{X}_{n+1} = \begin{cases} j &\text{ \text{with probability} }f\left(\frac{V^j_n}{n}\right), \quad j \in [2k-1] \\ 2k &\text{ \text{with probability} }1-\sum_{j=1}^{2k-1}f\left(\frac{V^j_n}{n}\right), \quad j = 2k.\end{cases}\end{align} Then the elephant moves in the direction $\mathcal{X}_{n+1}$ with probability $p \in (0,1)$, or moves in one of the remaining $2k-1$ directions with probability $(1-p)/(2k-1)$.
So, for any $n\geq 0$, the location $\boldsymbol{S}_n$ of the elephant at time $n$ is given by
\[
    \boldsymbol{S}_{n+1} = \boldsymbol{S}_{n} + \boldsymbol{X}_{n+1},
\]
where $\boldsymbol{S}_0 = \boldsymbol{0}_k$, \begin{align*}\boldsymbol{X}_{1} = \boldsymbol{u}_j \text{ with probability } \frac{1}{2k}, \quad 1 \leq j \leq 2k,\end{align*} and and $\boldsymbol{X}_{n+1}$, $n \geq 1$, are as follows. For $n \geq 1$, $V^j_{n} = \sum_{i=1}^n \boldsymbol{1}\{\boldsymbol{X}_i = \boldsymbol{u}_{j}\}$, $1 \leq j \leq 2k-1$ and $\mathcal{X}_{n+1}$ is given by~\eqref{eq:chi_n}.
Given $\mathcal{X}_{n+1}$,
\begin{align*}
   \boldsymbol{X}_{n+1} = \begin{cases} \boldsymbol{u}_{\mathcal{X}_{n+1}} &\text{ w.p. } p, \\ \boldsymbol{u}_{j}
&\text{ w.p. } \frac{1-p}{2k-1}, \quad 1 \leq j \leq 2k, j \neq \mathcal{X}_{n+1}, \end{cases} \quad 0 < p < 1.
\end{align*} 
This model reduces to the $k$-dimensional elephant random walk of \cite{bercu2019multi} for $f(x) = x$.\\\\
By taking $\boldsymbol{\widetilde{S}_n} = \begin{pmatrix}
    V^1_n & V^2_n & \ldots & V^{2k-1}_n
\end{pmatrix}^{\top}$, this can be shown to be a multidimensional generalized elephant random walk model with the parameters $s = 2k-1$, $d = k$, $r = 2k$, $\boldsymbol{Y}_1 \stackrel{d}{=} \delta_{\boldsymbol{1}_{2k-1}}$, $\boldsymbol{\Pi}^{2k-1}_{2k} = \{\{1\}, \{2\}, \ldots, \{2k-1\},\emptyset\}$, {\[\boldsymbol{A}_{k\times 2k-1} = \begin{bmatrix} 1 & -1 & 0 & 0 & 0 & 0 &\ldots & 0 & 0 & 0 \\ 0 & 0 & 1 & -1 & 0 & 0 & \ldots & 0 & 0 & 0\\ \vdots & \vdots & \vdots & \vdots & \vdots & \vdots & \vdots & \vdots & \vdots & \vdots \\ 0 & 0 & 0 & 0 & 0 & 0 & \ldots & 1 & -1 & 0 \\ 1 & 1 & 1 & 1 & 1 & 1 & \ldots &1 & 1 & 2 \end{bmatrix}, \quad \boldsymbol{b}_{k\times 1} = \begin{bmatrix}
    0 \\ 0 \\ \vdots \\ 0 \\ -1 
\end{bmatrix},\]} and, $\boldsymbol{\mathcal{P}} : [0,1]^{2k-1} \to [0,1]^{2k-1}$ given by, 
\begin{align*}{\mathcal{P}}_j(\boldsymbol{x}) = pf(x_j) + \frac{1-p}{2k-1}(1-f(x_j)), \quad 1 \leq j \leq 2k-1.\end{align*}

\begin{remark}
    Like the one-dimensional generalized elephant random walk, there is an asymmetry in the dynamics of the $k$-dimensional generalized elephant random walk. Namely, though the probability of choosing the $j$-th direction at epoch $n + 1$ is equal to $f(V^j_n/n)$ for $1 \leq j \leq 2k-1$, the same for choosing the $2k$-th direction at epoch $n + 1$, in general, is not equal to $f(V^{2k}_n/n)$, where, $V^{2k}_n = n - \sum_{j=1}^{2k-1}V^j_n$ is the number of steps in the $2k$-th direction till time $n$. We call the $k$-dimensional generalized elephant random walk symmetric if for any $x_1, \ldots, x_{2k-1} \in [0,1]$ satisfying $\sum_{i=1}^{2k-1}x_i < 1$, $f$ satisfies \[\sum_{i=1}^{2k-1}f(x_i)+f\left(1-\sum_{i=1}^{2k-1}x_i\right)  = 1.\] It is easy to see that for $k =2$, the $k$-dimensional generalized elephant random walk with a continuous $f$ is symmetric if and only if $f$ is of the form $f(x) = ax+b$ with $a+2kb = a+4b = 1$.
\end{remark}

\begin{remark}
    The multidimensional generalized elephant random walk is indeed a generalized model in the sense that random walk models that are a combination of two or more of the above models (for example, we can consider the $k$-dimensional generalized minimal random walk as a combination of Model~\ref{eg:gmrw} and Model~\ref{eg:gmerw}) can also be described (with appropriate parameters) and analyzed using our model. 
\end{remark}

\subsection{Main Results}\label{sec:main-results}

Define $\boldsymbol{\mu} := \mathbb{E}\left(\boldsymbol{Y}_1\right)$ and $\mathcal{P}_r : D_s \mapsto [0,1]$ be given by $\mathcal{P}_r(\boldsymbol{x}) =  1-\sum_{j=1}^{r-1}\mathcal{P}_j(\boldsymbol{x})$. The behavior of the multidimensional generalized elephant random walk is determined by the function $\boldsymbol{H}: D_s \mapsto [0,\infty)^s$, given by, 
\begin{align}\label{eq:H}
\boldsymbol{H}(\boldsymbol{x}) &:= 
\sum_{i=1}^{r}\mathcal{P}_i(\boldsymbol{x})\boldsymbol{\mu}^{(\pi_i)} 
= \begin{pmatrix}
    \mathcal{P}_1(\boldsymbol{x})\boldsymbol{\mu}_{(\pi_1)} \\ \vdots \\ \mathcal{P}_{r}(\boldsymbol{x})\boldsymbol{\mu}_{(\pi_{r})} 
\end{pmatrix}.
\end{align} 
\begin{remark}
    If $\pi_r = \emptyset$, then the vector representation of $\boldsymbol{H}$ in \eqref{eq:H} should be interpreted as 
\begin{align*}\begin{pmatrix}
    \mathcal{P}_1(\boldsymbol{x})\boldsymbol{\mu}_{(\pi_1)} \\ \vdots \\ \mathcal{P}_{r-1}(\boldsymbol{x})\boldsymbol{\mu}_{(\pi_{r-1})} 
\end{pmatrix}.
\end{align*}
\end{remark}
Our first result states a sufficient condition for almost sure convergence of the multidimensional generalized elephant random walk. 
\begin{theorem}[Almost sure convergence]\label{thm:slln-gerw}
Let $\mathbb{E}\left(\|\boldsymbol{Y}_1\|^2\right) < \infty$. Suppose there exists unique $\boldsymbol{x}_0 \in D_s$ such that $\boldsymbol{H}(\boldsymbol{x}_0) = \boldsymbol{x}_0$ 
and for any closed $C \subseteq D_s\setminus \{\boldsymbol{x}_0\}$, 
\begin{align}\label{eq33}
    \sup_{\boldsymbol{x} \in C} \left\{(\boldsymbol{x}-\boldsymbol{x}_0)^{\top}\left(\boldsymbol{H}(\boldsymbol{x}) - \boldsymbol{x}\right)\right\} < 0.
\end{align}
Then, 
\begin{align}\label{eq:asconv1}
   \frac{\boldsymbol{S}_n}{n} \stackrel{a.s.}{\to} \boldsymbol{A}\boldsymbol{x}_0 + \boldsymbol{b}.
\end{align}
\end{theorem}

\begin{remark}  
Note that if $\boldsymbol{H}$ is continuous on $D_s$ and 
\[
   (\boldsymbol{x}-\boldsymbol{x}_0)^{\top}\left(\boldsymbol{H}(\boldsymbol{x}) - \boldsymbol{x}\right) < 0 \text{ for all } \boldsymbol{x} \in D_s \setminus \{\boldsymbol{x}_0\},
\]
then \eqref{eq33} automatically holds for all closed set $C \subseteq D_s\setminus \{\boldsymbol{x}_0\}$.    
\end{remark}

If ${\boldsymbol{S}_n}/{n}$ almost surely converges to $\boldsymbol{A}\boldsymbol{x}_0 + \boldsymbol{b}$ for some non-random $\boldsymbol{x}_0 \in {D}_s$, then the higher-order behavior of the multidimensional generalized elephant random walk depends on $J_{\boldsymbol{H}}(\boldsymbol{x}_0)$, the Jacobian matrix of $\boldsymbol{H}$ at $\boldsymbol{x}_0$ (assuming $\boldsymbol{H}$ to be differentiable in a neighborhood of $\boldsymbol{x}_0$). Similar to its one-dimensional counterpart, a phase transition may happen depending on the nature of $J_{\boldsymbol{H}}(\boldsymbol{x}_0)$. First, we need the following assumptions.

\begin{assumption}\label{clt:assump1}
    Let $\boldsymbol{H}$ be differentiable %
    around $\boldsymbol{x}_0$ and \[J_{\boldsymbol{H}}(\boldsymbol{x}_0) = \left(\left(\frac{\partial \boldsymbol{H}_i}{\partial \boldsymbol{x}_j}\mid _{\boldsymbol{x} = \boldsymbol{x}_0}\right)\right)_{ij}\] be the Jacobian matrix of $\boldsymbol{H}$ at $\boldsymbol{x}_0$. Let $J^{\top}_{\boldsymbol{H}}(\boldsymbol{x}_0)$ have $l$ blocks $\boldsymbol{J}_1, \ldots, \boldsymbol{J}_l$ in its Jordan canonical form, $\lambda_1, \ldots, \lambda_l$ be the diagonal elements of the respective Jordan blocks and $\kappa_1, \ldots, \kappa_l$ be the corresponding block-sizes.
    Further, we assume that $\tau := \max \{ \text{Re}(\lambda_j) : 1 \leq j \leq l \} < 1$, and denote $\kappa := \max \left\{\kappa_j : \operatorname{Re}\left(\lambda_j\right) = \tau\right\}$.
\end{assumption}

\begin{remark}
    Let $J^{\top}_{\boldsymbol{H}}(\boldsymbol{x}_0) = \boldsymbol{R}\boldsymbol{J}\boldsymbol{Q}^{\top}$ be the Jordan decomposition of $J^{\top}_{\boldsymbol{H}}(\boldsymbol{x}_0)$ with $\boldsymbol{Q}^{\top} = \boldsymbol{R}^{-1}$. Further, assume 
    \begin{align*}
        \boldsymbol{J} = \begin{bmatrix}
            \boldsymbol{J}_1 & 0 & \ldots & 0 \\ 0 & \boldsymbol{J}_2 & \ldots & 0 \\ \vdots & \vdots & \vdots & \vdots \\ 0 & \ldots & 0 & \boldsymbol{J}_l 
        \end{bmatrix} \quad \text{with} \quad \boldsymbol{J}_j = \begin{bmatrix}
            \lambda_j & 1 & 0 & \ldots & 0 \\ 0 & \lambda_j & 1 & \ldots & 0 \\ \vdots & \vdots & \vdots & \vdots \\ 0 & \ldots & 0 & 0 & \lambda_j
        \end{bmatrix}_{\kappa_j \times \kappa_j}, \quad 1 \leq j \leq l.
    \end{align*}
    We denote corresponding block-decompositions of $\boldsymbol{R}$ and $\boldsymbol{Q}$ in the following way:
    \[
       \boldsymbol{R} = \begin{bmatrix}
            \boldsymbol{R}_1 & \boldsymbol{R}_2 & \ldots & \boldsymbol{R}_l\end{bmatrix}, \quad \boldsymbol{Q} = \begin{bmatrix}
            \boldsymbol{Q}_1 & \boldsymbol{Q}_2 & \ldots & \boldsymbol{Q}_l\end{bmatrix},
    \]
    where each $\boldsymbol{R}_j$ and $\boldsymbol{Q}_j$ have $\kappa_j$ columns for $j \in [l]$. For $j \in [l]$, %
    we denote the first (respectively, the last or $\kappa_j$-th) %
    column of $\boldsymbol{R}_j$ (respectively, $\boldsymbol{Q}_j$) by $\boldsymbol{R}^*_{j}$ (respectively, $\boldsymbol{Q}^*_{j}$). Note that $\boldsymbol{R}^*_{j}$ (respectively, $\boldsymbol{Q}^*_{j}$) is the right (respectively, the left) eigenvector in the $j$-th block.
\end{remark}

\begin{remark}\label{rem:clt:assump1-1d}
    If $s = 1$ (which automatically implies $d=1$), then Assumption~\ref{clt:assump1} is equivalent to assuming that $H$ is differentiable in a neighborhood of $x_0$ with $\tau = H'(x_0) < 1$. Also, in this case, $\kappa = 1$.
\end{remark}

Note that, assumption~\ref{clt:assump1} implies that
\[
\boldsymbol{H}(\boldsymbol{x}) = \boldsymbol{H}(\boldsymbol{x}_0) + J_{\boldsymbol{H}}\left(\boldsymbol{x}_0\right)\left(\boldsymbol{x} - \boldsymbol{x}_0\right) + o\left(\|\boldsymbol{x} - \boldsymbol{x}_0\|\right) \quad \text{as } \boldsymbol{x} \to \boldsymbol{x}_0.
\]
However, we need a slightly stronger condition on $\boldsymbol{H}$ for certain results.
\begin{assumption}\label{clt:assump2} For some $\delta > 0$, we have
    \[\boldsymbol{H}(\boldsymbol{x}) = \boldsymbol{H}(\boldsymbol{x}_0) + J_{\boldsymbol{H}}\left(\boldsymbol{x}_0\right)\left(\boldsymbol{x} - \boldsymbol{x}_0\right)  + o\left(\|\boldsymbol{x} - \boldsymbol{x}_0\|^{1+\delta}\right) \quad \text{as } \boldsymbol{x} \to \boldsymbol{x}_0.\] 
\end{assumption}

Finally, denote $\boldsymbol{\Sigma} = \mathbb{E}\left(\boldsymbol{Y}_1\boldsymbol{Y}^{\top}_1\right)$ and 
\begin{align}\label{eq:sigma}
   \boldsymbol{\Sigma}_0 := \sum_{i=1}^{r}\mathcal{P}_i(\boldsymbol{x}_0)\boldsymbol{\Sigma}^{(\pi_i)}
   - \boldsymbol{x}_0\boldsymbol{x}^{\top}_0 = \begin{bmatrix}
       \mathcal{P}_1(\boldsymbol{x}_0)\boldsymbol{\Sigma}_{(\pi_1)} && \ldots && \boldsymbol{0} 
       \\ \vdots && \ddots && \vdots
       \\ \boldsymbol{0} && \ldots && \mathcal{P}_r(\boldsymbol{x}_0)\boldsymbol{\Sigma}_{(\pi_r)}
   \end{bmatrix} - \boldsymbol{x}_0\boldsymbol{x}^{\top}_0.
\end{align}

\begin{remark}
If $\pi_r = \emptyset$, then the matrix representation of $\boldsymbol{\Sigma}_0$ in \eqref{eq:sigma} should be interpreted as 
\begin{align*}\begin{bmatrix}
       \mathcal{P}_1(\boldsymbol{x}_0)\boldsymbol{\Sigma}_{(\pi_1)} && \ldots && \boldsymbol{0} 
       \\ \vdots && \ddots && \vdots
       \\ \boldsymbol{0} && \ldots && \mathcal{P}_{r-1}(\boldsymbol{x}_0)\boldsymbol{\Sigma}_{(\pi_{r-1})}
   \end{bmatrix} - \boldsymbol{x}_0\boldsymbol{x}^{\top}_0.
\end{align*}
\end{remark}

The following result depicts the behavior of ﬂuctuations of the walk around the almost sure
limit of Theorem~\ref{thm:slln-gerw}. 

\begin{theorem}[Fluctuations around the almost sure limit]\label{thm:clt-gerw}
Let the assumptions of Theorem~\ref{thm:slln-gerw} hold so that almost surely ${\boldsymbol{S}_n}/{n}$ converges to $\boldsymbol{A}\boldsymbol{x}_0 + \boldsymbol{b}$ for some $\boldsymbol{x}_0 \in {D}_s$. Moreover, let Assumption~\ref{clt:assump1} hold. Then we have the following.

\begin{enumerate}
\item[a)] Diffusive regime: If $\tau < {1}/{2}$, 
\begin{align}\label{eq:diffconv}
    \sqrt{n}\left(\frac{\boldsymbol{S}_n}{n} - \boldsymbol{A}\boldsymbol{x}_0 - \boldsymbol{b}\right) \stackrel{d}{\to} N\left(0,\boldsymbol{A}\boldsymbol{\Sigma}_1\boldsymbol{A}^{\top}\right),
\end{align}
where,
\begin{align}\label{eq:sigma1}
   \boldsymbol{\Sigma}_1 = \int_{0}^{\infty}e^{\left(J_{\boldsymbol{H}}\left(\boldsymbol{x}_0\right) -\mathbb{I}_s/2\right)u}\boldsymbol{\Sigma}_0\left(e^{\left(J_{\boldsymbol{H}}\left(\boldsymbol{x}_0\right) -\mathbb{I}_s/2\right)u}\right)^{\top} du.
\end{align}
\end{enumerate}

For $\tau \geq 1/2$, let additionally Assumption~\ref{clt:assump2} hold.

\begin{enumerate}
\item[b)] Critical regime: If $\tau = {1}/{2}$, 
\begin{align}\label{eq:cricconv}
    \frac{\sqrt{n}}{(\log n)^{{\kappa} - 1/2}}\left(\frac{\boldsymbol{S}_n}{n} - \boldsymbol{A}\boldsymbol{x}_0 - \boldsymbol{b}\right) \stackrel{d}{\to} N\left(0,\boldsymbol{A}\boldsymbol{\Sigma}_2\boldsymbol{A}^{\top}\right),
\end{align}
where,
\begin{align}\label{eq:sigma2}
   \boldsymbol{\Sigma}_2 = \lim_{n\to\infty}\frac{1}{(\log n)^{2\kappa - 1}}\int_{0}^{\log n}e^{\left(J_{\boldsymbol{H}}\left(\boldsymbol{x}_0\right) -\mathbb{I}_s/2\right)u}\boldsymbol{\Sigma}_0\left(e^{\left(J_{\boldsymbol{H}}\left(\boldsymbol{x}_0\right) -\mathbb{I}_s/2\right)u}\right)^{\top} du.
\end{align}
\item[c)] Supercritical regime: If $1/2 < \tau < 1$, 
then there exist random variables ${\xi}_1, {\xi}_2, \ldots$ such that
\begin{align}\label{eq:sdiffconv}
    \frac{n^{1-\tau}}{(\log n)^{{\kappa} - 1}}\left(\frac{\boldsymbol{S}_n}{n} - \boldsymbol{A}\boldsymbol{x}_0 - \boldsymbol{b}\right) - \sum_{j : \operatorname{Re}\left(\lambda_j\right) = \tau, \kappa_j = \kappa}e^{\iota\operatorname{Im}\left(\lambda_j\right)\log n}{\xi}_j\boldsymbol{A}\boldsymbol{R}^*_{j} \stackrel{a.s.}{\to} \boldsymbol{0}.
\end{align}
\end{enumerate}
\end{theorem}

\begin{remark}\label{rem:new3}
    Using the formula just before Theorem 2.2 of \cite{zhang2016central}, we get the following alternate form of $\boldsymbol{\Sigma}_2$:
    \[
      \boldsymbol{\Sigma}_2 = \frac{1}{((\kappa-1)!)^2(2\kappa-1)}\sum_{\substack{j_1, j_2 : \lambda_{j_1 } = \lambda_{j_1}, \\ \operatorname{Re}\left(\lambda_{j_1}\right) = 1/2, \kappa_{j_1} = \kappa_{j_2} = \kappa}} \boldsymbol{Q}^*_{j_1} \left(\boldsymbol{R}^*_{{j_1}}\right)^{\top} \boldsymbol{\Sigma}_0 \boldsymbol{R}^*_{{j_2}}\left(\boldsymbol{Q}^*_{{j_2}}\right)^{\top}.
    \]
\end{remark}

\begin{remark}\label{rem:sd}
 If all the eigenvalues $\lambda_j$ of $J_{\boldsymbol{H}}\left(\boldsymbol{x}_0\right)$ are real (this is the case when $J_{\boldsymbol{H}}\left(\boldsymbol{x}_0\right)$ is real symmetric, e.g. in Models~\ref{eg:ldrw}, \ref{eg:gmrw}, \ref{eg:gerwrs} and \ref{eg:gmerw}) and $1/2 < \tau < 1$, then we conclude that for the {random vector $\boldsymbol{\xi} = \sum_{j:\lambda_j=\tau, \kappa_j=\kappa} \xi_j \boldsymbol{R}^*_j$} such that
 \begin{align}\label{eq:re-eig}
    \frac{n^{1-\tau}}{(\log n)^{{\kappa} - 1}}\left(\frac{\boldsymbol{S}_n}{n} - \boldsymbol{A}\boldsymbol{x}_0 - \boldsymbol{b}\right) \stackrel{a.s.}{\to}\boldsymbol{A}{\boldsymbol{\xi}}.
\end{align}
In particular, when $s = 1$ and hence $\kappa = 1$ (e.g., Models~\ref{eg:ldrw} and \ref{eg:gmrw}), then 
\begin{align}\label{eq:re-eig1}
   {n^{1-\tau}}\left(\frac{{S}_n}{n} -{A}{x}_0 - {b}\right) \stackrel{a.s.}{\to}{A}{\xi}.
\end{align}
\end{remark}

\subsubsection{Some Additional Results in the Case $s=1$} When the underlying auxiliary random walk $(\boldsymbol{\widetilde{S}}_n)$ is one-dimensional (i.e. $s = 1$, e.g., Models~\ref{eg:ldrw} and \ref{eg:gmrw}), Theorems~\ref{thm:lil-gerw-cric} and \ref{thm:super-dev1m} describe the rate of almost sure convergences in \eqref{eq:asconv1} and \eqref{eq:re-eig1}, respectively. 

\begin{theorem}[Law of iterated logarithm]\label{thm:lil-gerw-cric}
    Let $s = 1$, $\mathbb{E}\left({Y}_1^{2+\epsilon}\right) < \infty$ for some $\epsilon > 0$ and the assumptions of Theorem~\ref{thm:slln-gerw} hold so that almost surely ${{S}_n}/{n}$ converges to ${A}{x}_0 + {b}$ for some ${x}_0 \in {D}_1$. Moreover, let Assumption~\ref{clt:assump1} (in particular, Remark~\ref{rem:clt:assump1-1d}) hold with $\tau \leq 1/2$. Then the following hold.
\begin{enumerate}
\item[a)] Diffusive regime: if $\tau < 1/2$, we have almost surely
\begin{align}\label{eq:lil-gerw1}
    &\limsup_{n \to \infty}\sqrt{\frac{n}{2\log\log n}} \left(\frac{{S}_n}{n} - {A}{x}_0 - {b}\right) \nonumber
    \\ = - &\liminf_{n \to \infty}\sqrt{\frac{n}{2\log\log n}} \left(\frac{{S}_n}{n} - {A}{x}_0 - {b}\right) = A\sqrt{\frac{x_0\mu^{-1}{\Sigma}-x^2_0}{1-2\tau}}. 
\end{align}
\item[b)] Critical regime: if $\tau = 1/2$ and ${H}$ is further twice differentiable at ${x}_0$, we have almost surely
\begin{align}\label{eq:lil-gerw2}
    &\limsup_{n \to \infty}\left(\frac{n}{2\log n\log\log\log n}\right)^{1/2} \left(\frac{{S}_n}{n} - {A}{x}_0 - {b}\right) \nonumber \\ = - &\liminf_{n \to \infty}\left(\frac{n}{2\log n\log\log\log n}\right)^{1/2} \left(\frac{{S}_n}{n} - {A}{x}_0 - {b}\right) = A\sqrt{{x_0\mu^{-1}{\Sigma}-x^2_0}}. 
\end{align}
\end{enumerate}
\end{theorem}

\begin{remark}\label{rem:new2}
    When $s = 1$, using \eqref{eq:H}, we have $H(x) = \mathcal{P}_1(x)\mu \leq \mu$ for all $x \in D_1$. Therefore, in this case, for all $x \in D_1$, $\mu H(x) \leq \mu^2 \leq \Sigma$ and so ${H(x)\mu^{-1}{\Sigma}-H(x)^2} \geq 0$ as $H(x) \geq 0$. In particular, ${x_0\mu^{-1}{\Sigma}-x^2_0} \geq 0$.
\end{remark}

\begin{theorem}\label{thm:super-dev1m}
Let $s = 1$, $\mathbb{E}\left({Y}_1^{2+\epsilon}\right) < \infty$ for some $\epsilon > 0$ and the assumptions of Theorem~\ref{thm:slln-gerw} hold so that almost surely ${{S}_n}/{n}$ converges to ${A}{x}_0 + {b}$ for some ${x}_0 \in {D}_1$. Moreover, let Assumption~\ref{clt:assump1} (in particular, Remark~\ref{rem:clt:assump1-1d}) hold with $1/2 < \tau < 1$. Further, assume that ${H}$ is $(m+1)$-times differentiable in a neighborhood of $x_0$ 
for some $m \geq 1$ and ${\xi}$ be as in \eqref{eq:re-eig1}.
Then there exist constants $b_1 = 1, b_2, \ldots, b_{m+1}$, recursively given by
\begin{align*}b_{j+1} &= -\frac{1}{j(1-\tau)}\sum_{i=2}^{j+1}\frac{H^{(i)}(x_0)}{i!}\sum_{(c_1, \ldots, c_i) \in \mathcal{P}_{i,j+1}}\nu_{(c_1, \ldots, c_i)}b_{c_1}\ldots b_{c_i}, \quad j = 1, \ldots, m,\nonumber\end{align*} 
where, for $1 \leq i \leq t$, $\mathcal{P}_{i,t}$ and $\nu_{(c_1, \ldots, c_i)}$ are given by \eqref{eq:combn}, 
such that the following hold.
\begin{enumerate}
\item[a)] If \[ m \geq \frac{\tau - 1/2}{1-\tau}, \text{ then for } m_0 = \left\lfloor\frac{\tau - 1/2}{1-\tau}\right\rfloor,\]
almost surely,
\begin{align} \label{eq:slil-gerw}
    &\quad\quad\limsup_{n\to\infty}\sqrt{\frac{n}{2\log \log n}}\left(\left(\frac{{S}_n}{n} - {A}{x}_0 - {b}\right) - {A}\sum_{j = 0}^{m_0} b_{j+1}\left(\frac{\xi}{n^{1-\tau}}\right)^{j+1}
    \right) \nonumber\\
    &=-\liminf_{n\to\infty}\sqrt{\frac{n}{2\log \log n}}\left(\left(\frac{{S}_n}{n} - {A}{x}_0 - {b}\right) - {A}\sum_{j = 0}^{m_0} b_{j+1}\left(\frac{\xi}{n^{1-\tau}}\right)^{j+1} \right) \nonumber\\
   &= A\sqrt{\frac{{{x_0\mu^{-1}{\Sigma}-x^2_0}}}{2\tau -1}},
\end{align}
and 
\begin{align} \label{eq:sclt-gerw}
    \sqrt{n}\left(\left(\frac{{S}_n}{n} - {A}{x}_0 - {b}\right) - {A}\sum_{j = 0}^{m_0} b_{j+1}\left(\frac{\xi}{n^{1-\tau}}\right)^{j+1}\right) \stackrel{d}{\to} {N}\left(0, \frac{{A^2}\left({{x_0\mu^{-1}{\Sigma}-x^2_0}}\right)}{2\tau -1}\right).
\end{align}
\item[b)] If \[ m < \frac{\tau - 1/2}{1-\tau},\] 
then almost surely
\begin{align} \label{eq:s-gerw}
   \left(\frac{{S}_n}{n} - {A}{x}_0 - {b}\right) - {A}\sum_{j = 0}^{m} b_{j+1}\left(\frac{\xi}{n^{1-\tau}}\right)^{j+1} 
    = o\left({n}^{-(1-\tau){(m+1)}}\right)
\end{align}
\end{enumerate}
\end{theorem}

\begin{remark}%
    {In case of one-dimensional generalized elephant random walk of Section~\ref{eg:ldrw}, if we take $f(x) = x^2$, the walk is always diffusive. This is a special case of the generalized minimal random walk considered in Section~\ref{eg:gmrw} with $f(x)=x^2$ again, and further requiring $p=1-q$. However, if we considered the generalized minimal random walk with $f(x)=x^2$ and only requiring $p>q$, we can obtain richer asymptotic behavior. (A similar analysis can be carried out for $p<q$, which we leave out for the interested reader.) 
    
    From Theorem~\ref{thm:slln-gerw} we get \[\frac{S_n}{n} \stackrel{a.s.}{\to} s_{p,q} := \frac{1-\sqrt{1-4q(p-q)}}{2(p-q)}.\] Using Theorem~\ref{thm:clt-gerw}, we get that the walk is diffusive (respectively, critical, supercritical) if and only if $q(p-q) < 3/16$ (respectively, $q(p-q) = 3/16$, $q(p-q) > 3/16$). In particular, if $q(p-q) < {3}/{16}$, 
\begin{align*}
    \sqrt{n}\left(\frac{S_n}{n} - s_{p,q}\right) \stackrel{d}{\to} N\left(0,\frac{2q(p-q) - (1-(p-q))\left(1-\sqrt{1-4q(p-q)}\right)}{2(p-q)^2\left(2\sqrt{1-4q(p-q)}-1\right)}\right).
\end{align*}
However, if $q(p-q) = 3/16$, we have $s_{p,q} = \frac1{4(p-q)}$ and
\begin{align*}
    \sqrt{\frac{n}{\log n}}\left(\frac{S_n}{n} - s_{p,q}\right) \stackrel{d}{\to} N\left(0,\frac{4(p-q)-1}{16(p-q)^2}\right),
\end{align*}
and if $q(p-q) > 3/16$, there exists a finite random variable $W$ (that may depend on $p$, $q$ and $r$) such that
\begin{align*}
    n^{{\sqrt{1-4q(p-q)}}}\left(\frac{S_n}{n} - s_{p,q}\right) \stackrel{a.s.}{\to} W.
\end{align*}
Note that $q(p-q) \le p^2/4 <1/4$. Thus, for $3/16 < q(p-q) < 1/4$, choosing $k \geq 1$ satisfying \[\frac{1}{4} - \left(\frac{1}{4k}\right)^2 \leq q(p-q) < \frac{1}{4} - \left(\frac{1}{4k+4}\right)^2,\] an application of Theorem~\ref{thm:super-dev1m} (in particular, \eqref{eq:sclt-gerw}) yields,
\begin{align*}
    &\sqrt{n}\left(\left(\frac{{S}_n}{n} - s_{p,q}\right) - \sum_{j = 0}^{k-1} b_{j+1}\left(\frac{W}{n^{{\sqrt{1-4q(p-q)}}}}\right)^{j+1}\right) \\ &\stackrel{d}{\to} N\left(0,\frac{2q(p-q) - (1-(p-q))\left(1-\sqrt{1-4q(p-q)}\right)}{2(p-q)^2\left(1-2\sqrt{1-4q(p-q)}\right)}\right),
\end{align*}
where the sequence $\left(b_j\right)_{j \geq 1}$ is recursively defined as: 
\[
   b_1 = 1, \quad b_{j+1} = -\frac{p-q}{j\sqrt{1-4q(p-q)}} \sum_{l=1}^j b_lb_{j+1-l}, \quad j \geq 1.
\]}
\end{remark}

\section{Stochastic Approximation}\label{sec:sa}

In this section, we recall the tool of stochastic approximation, set up the relevant notations, and discuss some results, which are useful to prove our main Theorems.

\subsection{Stochastic Approximation: An Overview}\label{sec:sa-ldrw-gerw}

Let $\psi: \mathbb{R} \to \mathbb{R}$ be an unknown monotone (without loss of generality, increasing) function and we want to find a $\theta_0 \in \mathbb{R}$ that is a solution to the equation $\psi(\theta) = 0$. If we can observe the value of the function at all given points, then there are various
rapidly convergent methods available for solving this problem such as Newton's method. However, the situation is more involved when the functional value can only be observed with some random noise. To solve this problem, the following procedure was suggested in~\cite{robbins1951stochastic}. The method involves constructing a random sequence $\{\Theta_n\}_{n \geq 1}$ which is expected to converge to $\theta_0$ almost surely. At time $n = 1$, $\Theta_1$ is assigned some arbitrary value. Suppose at time $n \geq 2$, given that we have already constructed $\Theta_{n-1}$, we can only observe $\Lambda_{n-1} = \psi(\Theta_{n-1}) + \epsilon_n$, where $\epsilon_n$ is the random noise at time $n$. It is assumed that $\mathbb{E}\left(\epsilon_n \mid \Theta_1, \ldots, \Theta_{n-1}\right) = 0$ for all $n \geq 2$. Based on $\Lambda_{n-1}$, $\Theta_n$ is constructed as follows:
\begin{align}\label{eq:basic_sa}
     \Theta_n = \Theta_{n-1} - a_{n} \Lambda_{n-1},
\end{align}
where $\{a_n\}_{n\geq 2}$ is a sequence of positive numbers satisfying,
\begin{align}\label{eq:a_n}
    \sum_{n=2}^{\infty} a_n = \infty, \quad \sum_{n=2}^{\infty} a^2_n < \infty.
\end{align}
Note that the square summability condition above further gives $a_n\to0$.
This algorithm is known as stochastic approximation in the literature. The intuition behind the algorithm is as follows. Since \[\mathbb{E}\left(\Theta_n - \Theta_{n-1} \mid \Theta_1, \ldots, \Theta_{n-1}\right) = -a_n\psi(\Theta_{n-1}) \gtrless 0, \text{ accordingly as } \Theta_{n-1} \lessgtr \theta_0,\] the procedure, on average, forces the sequence $\Theta_n$ to move toward $\theta_0$. However, we must ensure that the jumps $\Theta_n - \Theta_{n-1}$ are damped, for otherwise, the sequence will oscillate around $\theta_0$, and that they do not decrease too rapidly. These are guaranteed by the conditions \eqref{eq:a_n}. Also, the square-summability condition of $\{a_n\}_{n\geq 2}$ in \eqref{eq:a_n} helps in obtaining an appropriate $L^2$-bounded martingale. It is shown in \cite{robbins1951stochastic} that $\Theta_n$ converges to $\theta_0$ almost surely under conditions \eqref{eq:a_n}, provided certain restrictions are imposed on~$\psi$. 

We can describe the above stochastic approximation algorithm as a stochastic process. Let $\left(\Omega, \mathcal{T}\right)$ be a measurable space with a filtration $\{\mathcal{T}_n\}_{n\geq 1}$. Let $\boldsymbol{\psi} : \mathcal{D} \to \mathbb{R}^l$ be a function (not necessarily monotone) on $\mathcal{D} \subseteq \mathbb{R}^l$. We say the adapted process $\left(\boldsymbol{\Theta}_n, \mathcal{T}_n\right)_{n \geq 1}$ is a stochastic approximation process on $\mathcal{D}$ if it satisfies the following recurrence relation
\begin{align}\label{eq:sa-main}
    \boldsymbol{\Theta}_{n+1} = \boldsymbol{\Theta_n} - a_n\left(\boldsymbol{\psi}(\boldsymbol{\Theta}_n) + \boldsymbol{\epsilon}_{n+1}\right), \quad n \geq 1, \quad \boldsymbol{\Theta}_1 \in \mathcal{D}, 
\end{align}
where $\{a_n\}_{n\geq 1}$ is a sequence of non-random positive numbers (called the step-size sequence) and the $\left\{\mathcal{T}_{n}\right\}_{n \geq 2}$-adapted process {$\left\{\boldsymbol{\epsilon}_{n}\right\}_{n \geq 2}$} is a random noise sequence. The function {$\boldsymbol{\psi}$} is known as the {regression function or the drift}. %
To study the convergence of stochastic approximation processes, we need %
further appropriate properties of the step-size sequence and the random noise sequence. %
Under appropriate assumptions, almost surely $\boldsymbol{\Theta_n}$ converges to $\boldsymbol{\theta}_0$ for some $\boldsymbol{\theta}_0 \in D$ satisfying $\boldsymbol{\psi}(\boldsymbol{\theta_0})=0$ (see, for example, Theorem 1 of Section 5 of \cite{MR1082341} (Part II)). Further, under additional assumptions, the limiting behavior of appropriately scaled $\boldsymbol{\Theta_n} - \boldsymbol{\theta}_0$ is known (see, for example, Theorem 2.1, Theorem 2.2 and Theorem 2.3 of \cite{zhang2016central}). In dimension $l = 1$, additional convergence properties of stochastic approximation processes can be derived which we consider in Section~\ref{sec:sa1d}.

\subsection{One-dimensional Stochastic Approximation Processes}\label{sec:sa1d}

Throughout this section, we assume that the stochastic approximation process is one-dimensional (i.e., $l=1$) and satisfies the following assumptions. 
\begin{assumption}\label{assum:n1}
    The step-size sequence $\{a_n\}_{n\geq 1}$ is given by $a_n = {1}/{(n+1)}$, there exist $\theta_0 \in \mathcal{D}$ with $\psi(\theta_0) = 0$ such that $\Theta_n \stackrel{a.s.}{\to}  \theta_0$, and the drift $\psi$ is differentiable at $\theta_0$ with $\psi'(\theta_0) >  0$.
\end{assumption}
Moreover, the noise sequence $\left({\epsilon}_{n}\right)_{n \geq 2}$ satisfy the following assumptions.
\begin{assumption}\label{assum:n2}
The noise sequence $\left({\epsilon}_{n}, \mathcal{T}_{n}\right)_{n \geq 2}$ is an $L^2$ martingale-difference sequence, such that %
for some $s > 0$, we have
\begin{align}\label{def:Sigma}
\lim_{n\to\infty}\mathbb{E}\left({\epsilon}_{n+1}^2 \mid \mathcal{T}_{n}\right) = s^2 %
\end{align}
and, for some $\beta > 0$, 
\begin{align*}%
        \sup_{n \geq 1} \mathbb{E}\left( |{\epsilon}_{n+1}|^{2+\beta} \mid \mathcal{T}_{n}\right) < \infty.
\end{align*}
\end{assumption}
Sometimes we need a little more than differentiability of $\psi$ at $\theta_0$.
\begin{assumption}\label{assum:n3}
    For some $\delta > 0$, we have
    \[\psi(\theta) = \psi(\theta_0) + \psi'(\theta_0)\left(\theta - \theta_0\right)  + o\left(|\theta - \theta_0|^{1+\delta}\right) \quad \text{as } \theta \to \theta_0.\] 
\end{assumption}
Then we have the following result regarding the fluctuation ${\Theta_n} - {\theta}_0$ of $\Theta_n$ around $\theta_0$.
\begin{theorem}\label{thm:sa1} Under Assumptions~\ref{assum:n1} and \ref{assum:n2}, ${\Theta_n} - {\theta}_0$ exhibit the following different behaviors, depending on the value of $\psi'(\theta_0)$.
\begin{itemize}
    \item[a)] If  $\psi'(\theta_0) > 1/2$, we have, almost surely,
    \begin{align} \label{eq:limsups1}
    &\limsup_{n\to \infty} \sqrt{\frac{n}{2\log\log n}}\left(\Theta_n - \theta_0\right) \nonumber %
    \\ = -&\liminf_{n\to \infty} \sqrt{\frac{n}{2\log\log n}}\left(\Theta_n - \theta_0\right) = \sqrt{\frac{s^2}{2 \psi'(\theta_0) - 1}},%
    \end{align}
    and
   \begin{align}\label{eq:diffclt26}
       \sqrt{n}\left(\Theta_n - \theta_0\right) \stackrel{d}{\to} N\left(0, \frac{s^2}{2 \psi'(\theta_0) - 1}\right),
    \end{align}
 where $s^2 > 0$ is given by~\eqref{def:Sigma}.
 
 \item[b)] If  $\psi'(\theta_0) = 1/2$ and further Assumption~\ref{assum:n3} holds, we have, almost surely,
    \begin{align}\label{eq:limsups2a}
    &\limsup_{n\to \infty} \left(\frac{n}{2\log n \log\log\log n}\right)^{1/2}\left(\Theta_n - \theta_0\right) \nonumber %
    \\ =-& \liminf_{n\to \infty} \left(\frac{n}{2\log n \log\log\log n}\right)^{1/2}\left(\Theta_n - \theta_0\right) = s,%
    \end{align}
    and
    \begin{align}\label{eq:critclt26}
       \sqrt{\frac{n}{\log n}}\left(\Theta_n - \theta_0\right) \stackrel{d}{\to} N\left(0, s^2\right),
    \end{align}
 where $s^2 > 0$ is given by~\eqref{def:Sigma}.
 \item[c)] If  $\psi'(\theta_0) < 1/2$ and further Assumption~\ref{assum:n3} holds, then there exists a finite random variable $Z$ such that
\begin{align}\label{eq:sdiffclt26}n^{\psi'(\theta_0)}\left({\Theta_n} - {\theta}_0\right) \stackrel{a.s.}{\to} Z.\end{align}
\end{itemize}
\end{theorem}

\begin{proof}
    If  $\psi'(\theta_0) > 1/2$, Assumptions~\ref{assum:n1} and~\ref{assum:n2} imply that $\{{\Theta}_n\}_{n \geq 1}$ satisfies all the assumptions of Theorem 1 of \cite{gaposkin1975law}. Also see the Remark before the quoted Theorem. Then using the result on the limit superior of certain scaled martingale in Lemma 2 of \cite{gaposkin1975law}, we obtain the value of the limit superior in~\eqref{eq:limsups1} from Theorem~1 of~\cite{gaposkin1975law}. %
Further, using the result on the limit inferior of the same scaled martingale in Lemma 2 of \cite{gaposkin1975law}, and arguing as in Theorem 1 of \cite{gaposkin1975law}, we obtain the value of the limit inferior in~\eqref{eq:limsups1}. %

In case $\psi'(\theta_0) = 1/2$, Assumptions~\ref{assum:n1}, \ref{assum:n2} and \ref{assum:n3} imply that $\{{\Theta}_n\}_{n \geq 1}$ satisfies all the assumptions of Theorem 1 of \cite{krasulina1996estimate}. Also see the Remark following the quoted Theorem. Again, using the result on the limit superior of certain scaled martingale in Lemma 2 of \cite{krasulina1996estimate}, we obtain the value of the limit superior in~\eqref{eq:limsups2a} from Theorem 1 of \cite{krasulina1996estimate}. %
Further, we get the result on the limit inferior of the said martingale from Theorem 1 of \cite{heyde}, and repeating the argument of Theorem~1 of~\cite{krasulina1996estimate} we obtain the limit inferior. %

{In a similar way, if $\psi'(\theta_0) > 1/2$ (respectively, $\psi'(\theta_0) = 1/2$, $\psi'(\theta_0) < 1/2$), \eqref{eq:diffclt26} (respectively, \eqref{eq:critclt26}, \eqref{eq:sdiffclt26}) follows from the display just after Equation (1) (respectively, Theorem 1, Theorem 2) of \cite{major1973limit}. These results also follow (see the proof of Theorem~\ref{thm:clt-gerw}) from Theorems 2.3, 2.1 and 2.2 of \cite{zhang2016central}, respectively.}
\end{proof}

If $\psi'(\theta_0) < 1/2$, then we can write \eqref{eq:sdiffclt26} equivalently, for $n\to\infty$, almost surely,
\begin{align}\label{eq:1-exp}
    {\Theta_n} - {\theta}_0 = \frac{Z}{n^{\psi'(\theta_0)}} + o\left(n^{-\psi'(\theta_0)}\right).
\end{align}
The exact rate of convergence in \eqref{eq:1-exp} is also known in the case $1/4 < \psi'(\theta_0) < 1/2$. 

\begin{theorem} \label{thm:weak0}
Let Assumptions~\ref{assum:n1} - \ref{assum:n3} hold. If further
$1/4 < \psi'(\theta_0) < 1/2$, then
\begin{equation} \label{eq:weak}
    \sqrt{n}\left(\Theta_n - \theta_0 - n^{-\psi'(\theta_0)}Z\right) \stackrel{d}{\to} N\left(0, \frac{s^2}{1-2\psi'(\theta_0)}\right),
\end{equation}
and, almost surely,
\begin{align}\label{eq:limsups2}
 &\limsup_{n\to \infty} \sqrt{\frac{n}{2\log\log n}}\left(\Theta_n - \theta_0 - \frac{Z}{n^{\psi'(\theta_0)}}\right) \nonumber %
\\ =-& \liminf_{n\to \infty} \sqrt{\frac{n}{2\log\log n}}\left(\Theta_n - \theta_0 - \frac{Z}{n^{\psi'(\theta_0)}}\right)=  \sqrt{\frac{s^2}{1-2 \psi'(\theta_0)}},%
\end{align}
 where $s^2 > 0$ is given by~\eqref{def:Sigma}.
\end{theorem}
\begin{proof}
{Proceeding as in the proof of Theorem~\ref{thm:sa1}, we get~\eqref{eq:weak} (respectively, \eqref{eq:limsups2}) from Theorem 3 of \cite{major1973limit} (respectively, Theorem 2 of \cite{krasulina1996estimate}). In Theorem 2 of \cite{krasulina1996estimate}, only the limit superior in~\eqref{eq:limsups2} is proved using the limit superior of certain scaled martingale sequence, given in Lemma 2 of \cite{krasulina1996estimate}. However, using the limit inferior of the corresponding martingale sequence (see Theorem 1 of \cite{heyde}), we get the limit inferior in \eqref{eq:limsups2}}.
\end{proof}

Note that we have the exact rate of convergence of $\Theta_n - \theta_0$ to $0$ for $\psi'(\theta_0) \ge 1/2$. For $\psi'(\theta_0)<1/2$, we can interpret~\eqref{eq:1-exp} as an expansion of order one for $\Theta_n - \theta_0$, while further refinement of the error term is given by Theorem~\ref{thm:weak0} for $1/4 < \psi'(\theta_0) < 1/2$. If we assume higher order differentiability of $\psi$, we can get analogous higher order expansion for $\Theta_n - \theta_0$. Such results are given by Theorem $(3k+1)$, Theorem $(3k+2)$ and Theorem $(3k+3)$ of \cite{major1973limit} as well as Theorem $(2k+1)$ and Theorem $(2k+2)$ of \cite{krasulina1996estimate}, where $k+1$ ($k \geq 0$) is the order of differentiability of $\psi$. However, \cite{krasulina1996estimate} and \cite{major1973limit} prove these theorems only for the cases $k = 0$ and $k = 1$.
To the best of our knowledge, a complete proof of these results for all $k > 1$ is missing in the literature. For the sake of completeness, we present an appropriately rephrased version of these results along with their proofs for all $k$. We begin with some notations. 
Let %
$\{b_t\}_{t\geq 1}$ be recursively given by (assuming appropriate differentiability of $\psi$)
\begin{equation} \label{eq: def b}
b_1 = 1, \quad b_t = \frac{1}{\psi'(\theta_0)(t-1)}\sum_{i=2}^{t}\frac{\psi^{(i)}(\theta_0)}{i!}\sum_{(c_1, \ldots, c_i) \in \mathcal{P}_{i,t}}\nu_{(c_1, \ldots, c_i)}b_{c_1}\ldots b_{c_i}, \quad t \geq 2,
\end{equation}
where for $1 \leq i \leq t$, $\mathcal{P}_{i,t}$ and $\nu_{(c_1, \ldots, c_i)}$ are given by \eqref{eq:combn}.
We also need the following lemmas.
\begin{lemma} Let Assumptions~\ref{assum:n1} - \ref{assum:n3} hold with $\psi'(\theta_0) < 1/2$ and $Z$ be as given in \eqref{eq:sdiffclt26}. Then, we have, as $n\to\infty$,
    \begin{align}\label{eq:supp1}
    n^{\psi'(\theta_0)}\left(\Theta_n - \theta_0\right) - b_1 Z &= A_n + B_n + O\left(n^{-1}\right),
\end{align}
where the sequences $(A_n)_{n \geq 1}$ and $(B_n)_{n \geq 1}$ are given by,
\begin{align}
    A_n &= \sum_{j=n}^{\infty}\frac{1}{{(j+1)}^{1-\psi'(\theta_0)}}\left(1+O\left(\frac{1}{j}\right)\right)R(\Theta_j - \theta_0),  \label{eq:suppno1} \\ B_n &= \sum_{j=n}^{\infty}\frac{1}{{(j+1)}^{1-\psi'(\theta_0)}}\left(1+O\left(\frac{1}{j}\right)\right)\epsilon_{j+1},\label{eq:suppn2}
\end{align}
with {$R(\theta) := \psi(\theta_0 +\theta) - \psi'(\theta_0)\theta$} and the bounds in~\eqref{eq:suppno1} and~\eqref{eq:suppn2} are uniform in $j$.
\end{lemma}
\begin{proof}
From~\eqref{eq:sa-main},we have for all $n,m \geq 1$,
\begin{align*}%
    &\left({\Theta}_{n+m+1} - \theta_0\right) \\ &= \left({\Theta_{n+m}} - \theta_0\right)\left(1 - \frac{\psi'(\theta_0)}{n+m+1}\right) - \frac{1}{n+m+1}\bigg( R(\Theta_{n+m}- \theta_0) + {\epsilon}_{n+m+1}\bigg)\\ 
    &= \left({\Theta_n} - \theta_0\right)\prod_{j=n}^{n+m}\left(1-\frac{\psi'(\theta_0)}{j+1}\right) \\ & \hspace{3cm}- \sum_{i=0}^m\frac{1}{n+1+i}\bigg(R(\Theta_{n+i}- \theta_0) + {\epsilon}_{n+i+1}\bigg)\prod_{j=n+i+1}^{n+m}\left(1-\frac{\psi'(\theta_0)}{j+1}\right).
\end{align*}
Using the fact that 
$\sum_{j=1}^{n}(1/j) = \gamma + \log n + O\left(1/n\right)$ as $n\to\infty$,
we have, for $-1 \leq i \leq m$,
\begin{align*}
    \log \prod_{j=n+i+1}^{n+m}\left(1-\frac{\psi'(\theta_0)}{j+1}\right) =
    & \sum_{j=n+i}^{n+m} \log \left(1-\frac{\psi'(\theta_0)}{j+1}\right)\\
    = & - \psi'(\theta_0) \sum_{j=n+i}^{n+m} \frac1{j+1} + 
    \sum_{j=n+i}^{n+m} O\left( \frac{1}{(j+1)^2}\right)\\
    = & - \psi'(\theta_0) \log \left(\frac{n+m+1}{n+i+1}\right) + O\left( \frac1{n+i}\right),
\end{align*}
giving
\[
   \prod_{j=n+i+1}^{n+m}\left(1-\frac{\psi'(\theta_0)}{j+1}\right) = \left(\frac{n+i+1}{n+m+1}\right)^{\psi'(\theta_0)}\left(1+O\left(\frac{1}{n+i}\right)\right),
\]
where the bound is uniform in $i$ and $m$. Hence
\begin{align*}%
    &{n}^{\psi'(\theta_0)}\left({\Theta_n} - \theta_0\right)\left(1+O\left(\frac{1}{n}\right)\right) - (n+m+1)^{\psi'(\theta_0)}\left({\Theta}_{n+m+1} - \theta_0\right) \\ =&
    \sum_{i=0}^m\left(\frac{1}{n+i+1}\right)^{1-\psi'(\theta_0)}\left(1+O\left(\frac{1}{n+i}\right)\right)\bigg(R(\Theta_{n+i}- \theta_0) + {\epsilon}_{n+i+1}\bigg) \end{align*}
{Letting $m \to \infty$} and using \eqref{eq:1-exp}, we have the required result.
\end{proof}

The following lemma immediately follows from Martingale Central Limit Theorem and Martingale Law of Iterated Logarithms (see the proofs of Theorem 3 of \cite{major1973limit} and Theorem 2 of \cite{krasulina1996estimate}, respectively). 
\begin{lemma} Assume that $\left(B_n\right)_{n\geq 1}$ is given by \eqref{eq:suppn2} and Assumptions~\ref{assum:n1} and \ref{assum:n2} hold with $\psi'(\theta_0) < 1/2$. Then, with $s^2>0$ given by~\eqref{def:Sigma}, we have
\begin{align}\label{eq:suppnn1}
   n^{1/2 - \psi'(\theta_0)}B_n \stackrel{d}{\to} N\left(0, \frac{s^2}{1-2\psi'(\theta_0)}\right),
\end{align}
and almost surely,
\begin{equation}\label{eq:suppnn2}
   \limsup_{n\to\infty}\frac{n^{1/2 - \psi'(\theta_0)}B_n}{\sqrt{2\log\log n}} \nonumber%
   =-\liminf_{n\to\infty}\frac{n^{1/2 - \psi'(\theta_0)}B_n}{\sqrt{2\log\log n}} = \sqrt{\frac{s^2}{1-2\psi'(\theta_0)}}.
\end{equation}
Consequently, if $k$ is such that $0 < \psi'(\theta_0) < \frac{1}{2(k+1)}$, then almost surely
\begin{align}\label{eq:supp7n}
   n^{k\psi'(\theta_0)}B_n = o(1).
\end{align}
\end{lemma}

We are now ready to consider $k$-th order expansion of $\Theta_n-\theta_0$ analogous to~\eqref{eq:1-exp}. We first introduce some notations. Consider  $\psi$ to be $k$-times differentiable. As $\psi(\theta_0) = 0$, we write 
\begin{align}\label{eq:supp0}
   R(\theta) &= C_{k}(\theta) + {U_{k}(\theta)}, \quad \text{with} \quad {U_{k}(\theta)} = O(\theta^{k+1}) \quad \text{as $\theta\to 0$ and} \nonumber\\
   C_{k}(\theta) &= \frac{\psi^{(2)}(\theta_0)}{2!}\theta^2 + \ldots + \frac{\psi^{(k)}(\theta_0)}{k!}\theta^{k}.
\end{align}

\begin{theorem}\label{thm:A1}
Let Assumptions~\ref{assum:n1} and~\ref{assum:n2} hold and $\psi$ be $k$-times differentiable for some $k \geq 1$. For $k=1$, we additionally assume Assumption~\ref{assum:n3}. 

Then, for $k>1$, when $0 < \psi'(\theta_0) < \frac{1}{2(k-1)}$ holds, we almost surely have, as $n\to\infty$,
\begin{align}
   n^{(k-1)\psi'(\theta_0)}A_n &= \sum_{t=2}^{k}n^{(k-t) \psi'(\theta_0)}b_t Z^{t} + o(1). \label{eq:supp6} \intertext{For $k\ge 1$, when further $0 < \psi'(\theta_0) < \frac{1}{2k}$ holds, we almost surely have, as $n\to\infty$,}
    \Theta_n - \theta_0 &= \sum_{j = 1}^{k} b_j \left(\frac{Z}{n^{\psi'(\theta_0)}}\right)^j + o\left(n^{-{k}\psi'(\theta_0)}\right), \label{eq:nn1}
\end{align}
where the coefficients $b_n$'s are given by~\eqref{eq: def b}, together with~\eqref{eq:combn}.
\end{theorem}
\begin{proof}
   We shall use induction on $k$. For the base case $k=1$, we need to show only~\eqref{eq:supp6}, which follows from \eqref{eq:1-exp} (Theorem 2 of \cite{major1973limit}). Next we shall prove both the induction statements~\eqref{eq:supp6} and~\eqref{eq:nn1} for $(k+1)$ with $(k+1)$-times differentiable $\psi$ with $k\geq 1$.
   
   We first establish~\eqref{eq:supp6} for $(k+1)$. We assume $0 < \psi'(\theta_0) < \frac{1}{2k}$ and~\eqref{eq:nn1} holds almost surely by the induction hypothesis. Note that
\begin{align}\label{eq:supp2}
    n^{k\psi'(\theta_0)}A_n &=  n^{k\psi'(\theta_0)}\sum_{j=n}^{\infty}\frac{C_{k+1}(\Theta_j - \theta_0)}{{(j+1)}^{1-\psi'(\theta_0)}} \nonumber \\ 
    & \qquad + O\left(n^{k\psi'(\theta_0)}\sum_{j=n}^{\infty}\frac{(\Theta_j - \theta_0)^{k+2}}{{j}^{1-\psi'(\theta_0)}}\right) + O\left(n^{k\psi'(\theta_0)}\sum_{j=n}^{\infty}\frac{(\Theta_j - \theta_0)^{2}}{{j}^{2-\psi'(\theta_0)}}\right).
\end{align}
Now, using \eqref{eq:nn1}, we get
\begin{multline}\label{eq:supp3}
n^{k\psi'(\theta_0)}\sum_{l=n}^{\infty}\frac{(\Theta_l - \theta_0)^{k+2}}{{l}^{1-\psi'(\theta_0)}} = n^{k\psi'(\theta_0)}\sum_{l=n}^{\infty}\frac{\left(\sum_{j = 1}^k b_j %
    Z^j l^{-j\psi'(\theta_0)} + o(l^{-k\psi'(\theta_0)})\right)^{k+2}}{{l}^{1-\psi'(\theta_0)}}
    \\ = O\left(n^{k\psi'(\theta_0)}\sum_{l=n}^{\infty}\frac{1}{{l}^{1 + (k+1)\psi'(\theta_0)}}\right)
    = O\left(\frac{1}{n^{\psi'(\theta_0)}}\right).
\end{multline}
Further, since $0<k \psi'(\theta_0) < 1/2 < 1$, we get
\begin{multline}\label{eq:supp4}
n^{k\psi'(\theta_0)}\sum_{l=n}^{\infty}\frac{(\Theta_l - \theta_0)^{2}}{{l}^{2-\psi'(\theta_0)}} = n^{k\psi'(\theta_0)}\sum_{l=n}^{\infty}\frac{\left(\sum_{j = 1}^k b_j 
Z^j l^{-j \psi'(\theta_0)} + o(l^{-k\psi'(\theta_0)})\right)^{2}}{{l}^{2-\psi'(\theta_0)}}
    \\ = O\left(n^{k\psi'(\theta_0)}\sum_{l=n}^{\infty}\frac{1}{{l}^{2+\psi'(\theta_0)}}\right)
    = O\left(\frac{1}{n^{\psi'(\theta_0)+ 1-k\psi'(\theta_0)}}\right) %
    = O\left(\frac{1}{n^{\psi'(\theta_0)}}\right).
\end{multline}
Finally, we have
\begin{align}
&n^{k\psi'(\theta_0)}\sum_{l=n}^{\infty}\frac{C_{k+1}(\Theta_l - \theta_0)}{{(l+1)}^{1-\psi'(\theta_0)}} = \sum_{i=2}^{k+1}\frac{\psi^{(i)}(\theta_0)}{i!}n^{k\psi'(\theta_0)}\sum_{l=n}^{\infty}\frac{(\Theta_l - \theta_0)^i}{{(l+1)}^{1-\psi'(\theta_0)}} \nonumber\\
    =& \sum_{i=2}^{k+1}\frac{\psi^{(i)}(\theta_0)}{i!}n^{k\psi'(\theta_0)}\sum_{l=n}^{\infty}\frac{\left(\sum_{j = 1}^k b_j Z^j l^{-j \psi'(\theta_0)}+ o(l^{-k\psi'(\theta_0)})\right)^i}{{l}^{1-\psi'(\theta_0)}} \left( 1+ O\left(\frac1l\right) \right) \label{eq:int1}\\
    =& \sum_{i=2}^{k+1}\frac{\psi^{(i)}(\theta_0)}{i!} \left[\sum_{t=i}^{k+1}\sum_{(c_1, \ldots, c_i) \in \mathcal{P}_{i,t}}\nu_{(c_1, \ldots, c_i)} b_{c_1}\ldots b_{c_i} Z^{t}n^{k\psi'(\theta_0)}
    \sum_{l=n}^{\infty}\frac{1}{{l}^{1 + (t-1) \psi'(\theta_0)}} %
    \right. \nonumber \\  & \hspace{1cm} + O\left( n^{k\psi'(\theta_0)} \sum_{l=n}^\infty \frac1{l^{2+\psi'(\theta_0}} \right) + \left. o\left(n^{k\psi'(\theta_0)}\sum_{l=n}^{\infty}\frac{1}{{l}^{1 + k \psi'(\theta_0)}}\right)\right] \nonumber 
    \\ =& \sum_{t=2}^{k+1}n^{k\psi'(\theta_0)} \sum_{l=n}^{\infty}\frac{1}{{l}^{1+(t-1)\psi'(\theta_0)}}%
    \psi'(\theta_0) (t-1) b_t Z^{t} + O\left(\frac{1}{n^{1-(k-1)\psi'(\theta_0)}}\right) + o\left(1\right) \nonumber \\ 
    =& \sum_{t=2}^{k+1} n^{(k+1-t)\psi'(\theta_0)}
    b_t Z^{t}
    + 
    O\left(\frac{1}{n^{1-(k-1)\psi'(\theta_0)}}\right) 
    + o\left(1\right) \label{eq:int2} \\ 
    =& \sum_{t=2}^{k+1}n^{(k+1-t)\psi'(\theta_0)}b_t Z^{t} +  o\left(1\right),  \label{eq:supp5}
\end{align}
where we use~\eqref{eq:nn1} in~\eqref{eq:int1}, the fact 
\begin{align} \label{eq:karamata error}
\sum_{l=n}^\infty l^{-(1+w)} = \frac1w n^{-w} + O\left(n^{-(1+w)}\right) \quad \text{for $w>0$}
\end{align}
in~\eqref{eq:int2} and the fact
$0\le (k-1) \psi'(\theta_0) < k \psi'(\theta_0) < 1/2 < 1$ for $k\ge 1$ in~\eqref{eq:supp5}. Combining~\eqref{eq:supp2},~\eqref{eq:supp3},~\eqref{eq:supp4} and~\eqref{eq:supp5} we get the induction statement~\eqref{eq:supp6} for $(k+1)$.

As $\psi$ is differentiable,~\eqref{eq:supp1} holds almost surely. We now further assume $0 < \psi'(\theta_0) < \frac1{2(k+1)}$,~\eqref{eq:supp7n} holds almost surely. Thus, using~\eqref{eq:supp1},~\eqref{eq:supp7n} and~\eqref{eq:supp6} and $\psi'(\theta_0) < {1}/{k}$, we also obtain the induction statement~\eqref{eq:nn1} for $(k+1)$.
\end{proof}

\begin{remark} \label{rem:2diff}
    In Theorem~\ref{thm:A1}, for $k\ge 2$, the function $\psi$ is at least twice differentiable and thus, Assumption~\ref{assum:n3} holds.
\end{remark}

The following result extends Theorem~\ref{thm:weak0}, which holds for $1/4 < \psi'(\theta_0) < 1/2$, and gives the exact rate of convergence in~\eqref{eq:nn1}.

\begin{theorem}\label{thm:A2}
Let Assumptions~\ref{assum:n1} and~\ref{assum:n2} hold and $\psi$ be $k$-times differentiable for some $k \geq 2$. %
If further $\frac{1}{2(k+1)} < \psi'(\theta_0) \leq \frac{1}{2k}$ holds, 
then
\begin{align}\label{eq:suppwc}
& \sqrt{n}\left(\Theta_n - \theta_0 - \sum_{j = 1}^{k}b_j \left(\frac{Z}{n^{\psi'(\theta_0)}}\right)^j\right) %
\stackrel{d}{\to} {N}\left(0, \frac{s^2}{1-2 \psi'(\theta_0)}\right),
\end{align}
and, almost surely
\begin{align}\label{eq:supplil}
 &\limsup_{n\to \infty} \sqrt{\frac{n}{2\log\log n}}\left(\Theta_n - \theta_0 - \sum_{j = 1}^{k} b_j \left(\frac{Z}{n^{\psi'(\theta_0)}}\right)^j\right) \nonumber %
\\ =-& \liminf_{n\to \infty} \sqrt{\frac{n}{2\log\log n}}\left(\Theta_n - \theta_0 - \sum_{j = 1}^{k} b_j \left(\frac{Z}{n^{\psi'(\theta_0)}}\right)^j\right) = \sqrt{\frac{s^2}{1-2 \psi'(\theta_0)}}.%
\end{align}
Here the coefficients $b_n$'s are given by~\eqref{eq: def b}, together with~\eqref{eq:combn}, while the limiting variance $s^2>0$ is given by~\eqref{def:Sigma}.
\end{theorem}
\begin{proof}

We first consider $\frac{1}{2(k+1)} < \psi'(\theta_0) < \frac{1}{2k}$ and thus~\eqref{eq:nn1} holds. However, we need a sharper estimate of the error term. Since $k\ge 2$, we rewrite~\eqref{eq:nn1} as
\begin{align}\label{eq:nn1 O}
    \Theta_n - \theta_0 = \sum_{j = 1}^{k} b_j \left(\frac{Z}{n^{\psi'(\theta_0)}}\right)^j + o\left(n^{-{k}\psi'(\theta_0)}\right) = \sum_{j = 1}^{k-1} b_j \left(\frac{Z}{n^{\psi'(\theta_0)}}\right)^j + O\left(n^{-{k}\psi'(\theta_0)}\right).
\end{align}
Since $\psi$ is $k$-times differentiable, arguing as in~\eqref{eq:supp3}, we get
\begin{multline} \label{eq:supp2 new}
    n^{(k-1) \psi'(\theta_0)}A_n =  n^{(k-1) \psi'(\theta_0)} \sum_{j=n}^{\infty} \frac{C_{k}(\Theta_j - \theta_0)}{{(j+1)}^{1-\psi'(\theta_0)}} \\ 
    + O\left(n^{(k-1) \psi'(\theta_0)} \sum_{j=n}^{\infty}\frac{(\Theta_j - \theta_0)^{k+1}}{{j}^{1-\psi'(\theta_0)}}\right) + O\left(n^{(k-1) \psi'(\theta_0)} \sum_{j=n}^{\infty}\frac{(\Theta_j - \theta_0)^{2}}{{j}^{2-\psi'(\theta_0)}}\right).
\end{multline}

Again arguing as in~\eqref{eq:supp5}, but now using~\eqref{eq:nn1 O}, and further using~\eqref{eq:karamata error}, we get
\begin{align}\label{eq:suppn5}
&n^{(k-1) \psi'(\theta_0)} \sum_{l=n}^{\infty} \frac{C_{k}(\Theta_l - \theta_0)}{{(l+1)}^{1-\psi'(\theta_0)}} %
\nonumber\\
= &\sum_{i=2}^{k}\frac{\psi^{(i)}(\theta_0)}{i!}n^{(k-1) \psi'(\theta_0)} \sum_{l=n}^{\infty}\frac{\left(\sum_{j = 1}^{k-1} b_j {Z}^j l^{-j \psi'(\theta_0)} + O(l^{-k\psi'(\theta_0)})\right)^i}{{(l+1)}^{1-\psi'(\theta_0)}} \nonumber
    \\  =& \sum_{i=2}^{k}\frac{\psi^{(i)}(\theta_0)}{i!}\left[\sum_{t=i}^{k}\sum_{(c_1, \ldots, c_i) \in \mathcal{P}_{i,t}}\nu_{(c_1, \ldots, c_i)}b_{c_1}\ldots b_{c_i}Z^{t}n^{(k-1)\psi'(\theta_0)}\sum_{l=n}^{\infty}\frac{{l}^{-t\psi'(\theta_0)}}{{(l+1)}^{1-\psi'(\theta_0)}}%
    \right. \nonumber \\  & \hspace{1cm} + \left. O\left(n^{(k-1)\psi'(\theta_0)}\sum_{l=n}^{\infty}\frac{{l}^{-(k+1)\psi'(\theta_0)}}{{(l+1)}^{1-\psi'(\theta_0)}}\right)\right] \nonumber \\
    =& \sum_{t=2}^{k}n^{(k-t) \psi'(\theta_0)} b_t Z^{t} +
    O\left(\frac{1}{n^{1-(k-2)\psi'(\theta_0)}}\right) 
    + O\left(\frac{1}{n^{\psi'(\theta_0)}}\right) \nonumber %
    \\ =& \sum_{t=2}^{k}n^{(k-t)\psi'(\theta_0)}b_t Z^{t} +  O\left(\frac{1}{n^{\psi'(\theta_0)}}\right) %
\end{align}
where in the last step we use $\psi'(\theta_0) < \frac1{2k} < \frac1{k-1}$.
Combining~\eqref{eq:supp2 new},~\eqref{eq:suppn5},~\eqref{eq:supp3}, and~\eqref{eq:supp4}, for $\frac{1}{2(k+1)} < \psi'(\theta_0) < \frac{1}{2k}$ we get,
\begin{align}\label{eq:supp6-1}
   n^{(k-1)\psi'(\theta_0)}A_n = \sum_{t=2}^{k}n^{(k-t)\psi'(\theta_0)}b_t Z^{t} + O\left({n^{-\psi'(\theta_0)}}\right).
\end{align}
Thus, for $\frac{1}{2(k+1)} < \psi'(\theta_0) < \frac{1}{2k}$, we obtain~\eqref{eq:suppwc} (respectively,~\eqref{eq:supplil}) from~\eqref{eq:supp1},~\eqref{eq:supp6-1} 
and~\eqref{eq:suppnn1} (respectively~\eqref{eq:suppnn2}).

Finally, in case $\psi'(\theta_0) = \frac{1}{2k} < \frac{1}{2(k-1)}$, using \eqref{eq:supp1}, \eqref{eq:supp6} and the fact $1/2 - k\psi'(\theta_0) = 0$, we get, almost surely,
\[
\sqrt{n} \left( \Theta_n - \theta_0 - \sum_{j=1}^k b_j \left( \frac{Z}{n^{\psi'(\theta_0}} \right)^t \right) = o(1) + n^{\frac12 - \psi'(\theta_0)} B_n + O \left( \frac1{n^{1+\psi'(\theta_0)}} \right).
\]
Further, using \eqref{eq:suppnn1} (respectively, \eqref{eq:suppnn2}) we get \eqref{eq:suppwc} (respectively, \eqref{eq:supplil}). 
\end{proof}

\begin{remark}
    As noted in Remark~\ref{rem:2diff}, the twice differentiability of the function $\psi$ also in Theorem~\ref{thm:A2} guarantees Assumption~\ref{assum:n3} and it is not assumed separately.
\end{remark}

\begin{remark}
    As also noted in Remark~\ref{rem:0}, if the function $\psi$ is differentiable of order $k= \left\lfloor \frac1{2\psi'(\theta_0)} \right\rfloor$, then we get the exact order in the expansion of $\Theta_n - \theta_0$ given in Theorem~\ref{thm:A2}. However, if the function $\psi$ is not smooth enough, we shall only have Theorem~\ref{thm:A1}.
\end{remark}

\subsection{Auxiliary Random Walk and Stochastic Approximation}\label{sec:proofs-lemmas}

The auxiliary random walk that we use in our model, namely $\boldsymbol{\widetilde{S}}_n$, after appropriate scaling, is a stochastic approximation process. For $n\geq 1$, as $\boldsymbol{\widetilde{S}}_{n+1} = \boldsymbol{\widetilde{S}}_n + \boldsymbol{\widetilde{X}}_{n+1}$, we have
\begin{align}\label{pf-eq2}
    \frac{\boldsymbol{\widetilde{S}}_{n+1}}{n+1} &= \frac{\boldsymbol{\widetilde{S}}_{n}}{n} + \frac{1}{n+1}\left(\boldsymbol{\widetilde{X}}_{n+1} - \frac{\boldsymbol{\widetilde{S}}_{n}}{n}\right) \nonumber \\ &
    = \frac{\boldsymbol{\widetilde{S}}_{n}}{n} + \frac{1}{n+1}\left(\boldsymbol{H}\left(\frac{\boldsymbol{\widetilde{S}}_{n}}{n}\right) - \frac{\boldsymbol{\widetilde{S}}_{n}}{n} + \boldsymbol{\widetilde{X}}_{n+1} - \boldsymbol{H}\left(\frac{\boldsymbol{\widetilde{S}}_{n}}{n}\right)\right),
\end{align}
where $\boldsymbol{H}$ is given by \eqref{eq:H}. With $D_s$ as in Section~\ref{subsec:model-gerw}, define, $\boldsymbol{\gamma}: D_s \mapsto \mathbb{R}^s$ by $\boldsymbol{\gamma}(\boldsymbol{x}) := \boldsymbol{x} - \boldsymbol{H}(\boldsymbol{x})$ and for $n \geq 1$, let $\boldsymbol{\Gamma}_{n} = {\boldsymbol{\widetilde{S}}_n}/{n}$, $\mathcal{G}_{n} = \sigma\{\boldsymbol{\widetilde{X}}_1, \ldots, \boldsymbol{\widetilde{X}}_n\}$, $\alpha_n = 1/(n+1)$ and $\boldsymbol{e}_{n+1} = \boldsymbol{H}\left({\boldsymbol{\widetilde{S}}_{n}}/{n}\right) - \boldsymbol{\widetilde{X}}_{n+1}$. From \eqref{pf-eq2}, we get for $n \geq 1$,
\begin{align}\label{eq-sa2}
   \boldsymbol{\Gamma}_{n+1} = \boldsymbol{\Gamma}_{n} - \alpha_n\left(\boldsymbol{\gamma}(\boldsymbol{\Gamma}_{n}) + \boldsymbol{e}_{n+1}\right).
\end{align}

Thus the process $\left(\boldsymbol{\Gamma}_{n}, \mathcal{G}_{n}\right)$ is a stochastic approximation process on ${D}_s$ with drift $\boldsymbol{\gamma}$, step size $\left\{\alpha_n\right\}_{n\geq 1}$ and random noise $\left\{\boldsymbol{e}_{n}\right\}_{n\geq 2}$. Observe that, $\{\alpha_n\}_{n\geq 1}$ satisfies 
\begin{align*}%
\sum_{n=1}^{\infty} \alpha_n = \infty, \quad \sum_{n=1}^{\infty} \alpha^2_n < \infty.
\end{align*}
Some desired properties of $\left\{\boldsymbol{e}_{n}\right\}_{n\geq 2}$ are described in Lemma~\ref{lem-noiseanddrift2}%
, which are required to prove our main results. %
Its proof heavily uses the notations defined in Section~\ref{sec:gerw} (in particular, Section~\ref{subsec:model-gerw} and Section~\ref{sec:main-results}) and Section~\ref{sec:sa-ldrw-gerw}. 

\begin{lemma}[]\label{lem-noiseanddrift2}
{Assume that $\mathbb{E}\left(\|\boldsymbol{Y}_1\|^2\right) < \infty$. Then the random noise $\left\{\boldsymbol{e}_{n}\right\}_{n\geq 2}$ of the %
process $\left(\boldsymbol{\Gamma}_{n}, \mathcal{G}_{n}\right)$, given by~\eqref{eq-sa2}, satisfies the following properties. 
\begin{enumerate}
\item{}\label{item2} The conditional distribution of $\boldsymbol{\widetilde{X}}_{n+1}$ given $\mathcal{G}_{n}$ is a measurable function of $\boldsymbol{\Gamma}_{n}$ alone and so the same holds for $\boldsymbol{e}_{n+1}$ as well. More precisely, for any bounded measurable function $g : D_s \to \mathbb{R}$, 
\begin{align}\label{eq:g}
   \mathbb{E}\left(g\left(\boldsymbol{\widetilde{X}}_{n+1}\right) \mid \mathcal{G}_{n}\right) = \sum_{i=1}^{r}\mathcal{P}_i\left(\boldsymbol{\Gamma}_{n}\right)\mathbb{E}\left(g\left(\boldsymbol{Y}^{(\pi_i)}\right)\right).
\end{align}
Moreover, $\left(\boldsymbol{e}_{n}\right)_{n \geq 2}$ is an $L^2$ martingale-difference sequence with respect to the filtration $\left(\mathcal{G}_{n}\right)_{n \geq 2}$. 

Define the functions $\sigma^2$ and $\boldsymbol{\Sigma}$ on ${D}_s$ as: \begin{align}\label{eq:lem-Sigma}
\sigma^2\left(\boldsymbol{x}\right) := \mathbb{E}\left(\|\boldsymbol{e}_{n+1}\|^2 \mid \boldsymbol{\Gamma}_{n} = \boldsymbol{x}\right), \quad \boldsymbol{\Sigma}\left(\boldsymbol{x}\right) := \mathbb{E}\left(\boldsymbol{e}_{n+1}\boldsymbol{e}^{\top}_{n+1} \mid \boldsymbol{\Gamma}_{n} = \boldsymbol{x}\right).
\end{align}
Then we have
\begin{align*}%
\boldsymbol{\Sigma}\left(\boldsymbol{x}\right) &= \sum_{i=1}^{r}\mathcal{P}_i\left(\boldsymbol{x}\right)\boldsymbol{\Sigma}^{(\pi_i)}
  -  \boldsymbol{H}\left(\boldsymbol{x}\right)\left(\boldsymbol{H}\left(\boldsymbol{x}\right)\right)^{\top},
\intertext{and}
\sigma^2\left(\boldsymbol{x}\right) &= \sum_{i=1}^{r}\mathcal{P}_i\left(\boldsymbol{x}\right)\operatorname{tr}\boldsymbol{\Sigma}^{(\pi_i)} - \left(\boldsymbol{H}\left(\boldsymbol{x}\right)\right)^{\top}\boldsymbol{H}\left(\boldsymbol{x}\right).
\end{align*}
For some constant $C > 0$ and for all $\boldsymbol{x} \in {D}_s$, we also have \begin{align*}%
\mathbb{E}\left(\|\boldsymbol{\gamma}(\boldsymbol{\Gamma}_n) + \boldsymbol{e}_{n+1}\|^2 \mid \boldsymbol{\Gamma}_n = \boldsymbol{x}\right) \leq \|\boldsymbol{\gamma}(\boldsymbol{x})\|^2 + \sigma^2(\boldsymbol{x}) \leq C(1 + \|\boldsymbol{x}\|^2).
\end{align*}

Additionally, if $\boldsymbol{\Gamma}_n \stackrel{a.s.}{\to} \boldsymbol{x}_0$ for some $\boldsymbol{x}_0 \in {D}_s$ and $\boldsymbol{\mathcal{P}}$ is continuous at $\boldsymbol{x}_0$, then
\begin{align}\label{eq:lem-Sigmaconv}
\sigma^2\left(\boldsymbol{\Gamma}_{n}\right) \stackrel{a.s.}{\to} \sigma^2\left(\boldsymbol{x}_0\right), \quad \boldsymbol{\Sigma}\left(\boldsymbol{\Gamma}_{n}\right) \stackrel{a.s.}{\to} \boldsymbol{\Sigma}\left(\boldsymbol{x}_0\right).
\end{align}
\item{}\label{item3} The following Lindeberg condition is satisfied: for all $\delta > 0$,
    \begin{align}\label{eq:lind}
    \frac{1}{n}\sum_{k=1}^{n}\mathbb{E}\left( \|\boldsymbol{e}_{k+1}\|^2 \mathbf{1}\left\{\|\boldsymbol{e}_{k+1}\| \geq \delta\sqrt{n}\right\}\mid \mathcal{G}_{k}\right) \stackrel{a.s.}{\to} 0.
    \end{align}
If we additionally have $\mathbb{E}\left(\|\boldsymbol{Y}_1\|^{2+\epsilon}\right) < \infty$ for some $\epsilon > 0$, then
\begin{align}\label{eq:sup}
        \sup_{n \geq 1} \mathbb{E}\left( \|\boldsymbol{e}_{n+1}\|^{2+\epsilon} \mid \mathcal{G}_{n}\right) < \infty.
\end{align}
\end{enumerate}}
\end{lemma}

\begin{proof} 
{%
The assertion that the conditional distribution of $\boldsymbol{\widetilde{X}}_{n+1}$ given $\mathcal{G}_{n}$ is a measurable function of $\boldsymbol{\Gamma}_{n}$ alone and \eqref{eq:g} follow from \eqref{eq12}. Next observe that $\|\boldsymbol{\widetilde{X}}_{n+1}\| \leq \|\boldsymbol{Y}_{n+1}\|$ and for any $\boldsymbol{x} \in {D}_s$,
\begin{align}\label{eq:h-ineq}
    \|\boldsymbol{H}(\boldsymbol{x})\| &= \left\|\sum_{i=1}^{r}\mathcal{P}_i(\boldsymbol{x})\boldsymbol{\mu}^{(\pi_i)}\right\| \leq \|\boldsymbol{\mu}\| \sum_{i=1}^{r} \mathcal{P}_i(\boldsymbol{x}) = \|\boldsymbol{\mu}\|. 
\end{align}
Thus, for any $n \geq 1$,
\begin{align}\label{eq:e-bound}
    \|\boldsymbol{e}_{n+1}\| &\leq \left\|\boldsymbol{H}\left({\boldsymbol{\Gamma}_n}\right)\right\| + \|\boldsymbol{\widetilde{X}}_{n+1}\| \leq \|\boldsymbol{\mu}\| + \|\boldsymbol{Y}_{n+1}\|,
\end{align}
and so
\[\mathbb{E}\left(\|\boldsymbol{e}_{n+1}\|\right) \leq \|\boldsymbol{\mu}\| +  \mathbb{E}\left(\|\boldsymbol{Y}_1\|\right)< \infty.\]}

Also, from 
\eqref{eq12}, it follows that 
\begin{align*}
    \mathbb{E}\left(\boldsymbol{e}_{n+1} \mid \mathcal{G}_{n}\right) = \mathbb{E}\left(\boldsymbol{H}({\boldsymbol{\Gamma}_n}) - \boldsymbol{\widetilde{X}}_{n+1} \Big| \boldsymbol{\widetilde{S}}_{n}\right) = \boldsymbol{H}({\boldsymbol{\Gamma}_n}) -  \sum_{i=1}^{r}\mathcal{P}_i(\boldsymbol{x})\boldsymbol{\mu}^{(\pi_i)} 
    = 0.
\end{align*}
Thus 
$\left(\boldsymbol{e}_{n}\right)_{n \geq 2}$ is a martingale-difference sequence with respect to $\left(\mathcal{G}_{n}\right)_{n \geq 2}$.
Using~\eqref{eq:h-ineq}, 
for any $n \geq 1$, we get
\begin{align}\label{eq:e-bound2}
    \|\boldsymbol{e}_{n+1}\|^2 &\leq 2\left(\left\|\boldsymbol{H}\left({\boldsymbol{\Gamma}_n}\right)\right\|^2 + \|\boldsymbol{\widetilde{X}}_{n+1}\|^2\right)
   \leq 2\left(\|\boldsymbol{\mu}\|^2 + \|\boldsymbol{Y}_{n+1}\|^2\right),
\end{align}
which makes the martingale difference $L^2$-bounded.

{From \eqref{eq12}, we further have
\begin{align*}
\boldsymbol{\Sigma}\left(\boldsymbol{x}\right) = \mathbb{E}\left(\boldsymbol{e}_{n+1}\boldsymbol{e}^{\top}_{n+1} \mid \boldsymbol{\Gamma}_{n} = \boldsymbol{x}\right) &= \mathbb{E}\left( \boldsymbol{\widetilde{X}}_{n+1}\boldsymbol{\widetilde{X}}^{\top}_{n+1} \Big| \boldsymbol{\Gamma}_{n} = \boldsymbol{x}\right) - 
\boldsymbol{H}\left(\boldsymbol{x}\right)\left(\boldsymbol{H}\left(\boldsymbol{x}\right)\right)^{\top}
        \\ &= \sum_{i=1}^{r}\mathcal{P}_i\left(\boldsymbol{x}\right)\boldsymbol{\Sigma}^{(\pi_i)}
        -  \boldsymbol{H}\left(\boldsymbol{x}\right)\left(\boldsymbol{H}\left(\boldsymbol{x}\right)\right)^{\top}.
\end{align*}
As $\mathbb{E}\left(\|\boldsymbol{e}_{n+1}\|^2 \mid \boldsymbol{\Gamma}_{n} = \boldsymbol{x}\right) = \operatorname{tr} \mathbb{E}\left(\boldsymbol{e}_{n+1}\boldsymbol{e}^{\top}_{n+1} \mid \boldsymbol{\Gamma}_{n} = \boldsymbol{x}\right)$,
we have 
\begin{align*}%
\sigma^2\left(\boldsymbol{x}\right) &= \operatorname{tr} \boldsymbol{\Sigma} (\boldsymbol{x}) 
 = \sum_{i=1}^{r}\mathcal{P}_i\left(\boldsymbol{x}\right)\operatorname{tr}\boldsymbol{\Sigma}^{(\pi_i)} - \left(\boldsymbol{H}\left(\boldsymbol{x}\right)\right)^{\top}\boldsymbol{H}\left(\boldsymbol{x}\right).
  \end{align*}}
Further, 
$\sigma^2\left(\boldsymbol{x}\right) \leq \sum_{i=1}^{r}\mathcal{P}_i\left(\boldsymbol{x}\right)\operatorname{tr}\boldsymbol{\Sigma}^{(\pi_i)} 
  \leq \mathbb{E}\left(\|\boldsymbol{Y}_1\|^2\right)$. %

{Also, for all $\boldsymbol{x} \in {D}_s$,
\begin{align*}
    \|\boldsymbol{\gamma}\left(\boldsymbol{x}\right)\|^2 = \|\boldsymbol{x}-\boldsymbol{H}(\boldsymbol{x})\|^2 \leq 2\left(\|\boldsymbol{x}\|^2 + \|\boldsymbol{H}(\boldsymbol{x})\|^2\right) \leq 2\|\boldsymbol{x}\|^2 + 2 \|\boldsymbol{\mu}\|^2.
\end{align*}
Thus, for all $\boldsymbol{x} \in {D}_s$, we get a constant $C > 0$ such that
\[
  \mathbb{E}\left(\|\boldsymbol{\gamma}(\boldsymbol{\Gamma}_n) + \boldsymbol{e}_{n+1}\|^2 \mid \boldsymbol{\Gamma}_n = \boldsymbol{x}\right) \leq \|\boldsymbol{\gamma}\left(\boldsymbol{x}\right)\|^2 +  \sigma^2\left(\boldsymbol{x}\right) \leq C(1+\|\boldsymbol{x}\|^2).
\]}%
{The fact that $\boldsymbol{\Gamma}_n \stackrel{a.s.}{\to} \boldsymbol{x}_0$ and the continuity of $\boldsymbol{\mathcal{P}}$ at $\boldsymbol{x}_0$ imply $\boldsymbol{\Sigma}\left(\boldsymbol{\Gamma}_n\right) \stackrel{a.s.}{\to} \boldsymbol{\Sigma}\left(\boldsymbol{x}_0\right)$. Since $\operatorname{tr}$ is continuous operator, $\boldsymbol{\Sigma}\left(\boldsymbol{\Gamma}_n\right) \stackrel{a.s.}{\to} \boldsymbol{\Sigma}\left(\boldsymbol{x}_0\right)$ implies ${\sigma^2}\left(\boldsymbol{\Gamma}_n\right) \stackrel{a.s.}{\to} {\sigma^2}\left(\boldsymbol{x}_0\right)$.}

Now fix $\delta > 0$. Then, using \eqref{eq:e-bound},~\eqref{eq:e-bound2} and finite second moment of $\|\boldsymbol{Y}_1\|$, we get,
\begin{align*}%
&\frac{1}{n}\sum_{k=1}^{n}\mathbb{E}\left( \|\boldsymbol{e}_{k+1}\|^2 \mathbf{1}\left\{\|\boldsymbol{e}_{k+1}\| \geq \delta\sqrt{n}\right\}\mid \mathcal{G}_{k}\right) \nonumber
\\ \leq& \frac{2}{n}\sum_{k=1}^{n}\mathbb{E}\left( \left(\|\boldsymbol{\mu}\|^2 + \|\boldsymbol{Y}_{k+1}\|^2\right) \mathbf{1}\left\{\|\boldsymbol{\mu}\| + \|\boldsymbol{Y}_{k+1}\| \geq \delta\sqrt{n}\right\}\mid \mathcal{G}_{k}\right) \nonumber
\\ =& 2\mathbb{E}\left( \left(\|\boldsymbol{\mu}\|^2 + \|\boldsymbol{Y}_{1}\|^2\right) \mathbf{1}\left\{\|\boldsymbol{\mu}\| + \|\boldsymbol{Y}_{1}\| \geq \delta\sqrt{n}\right\}\right) \to 0. \nonumber
\end{align*}
{The assumption that $\mathbb{E}\left(\|\boldsymbol{Y}_1\|^{2+\epsilon}\right) < \infty$, combined with \eqref{eq:e-bound}, prove \eqref{eq:sup}.}
\end{proof}

\begin{remark}
    When $s=1$, from Remark~\ref{rem:new2}, we have $H(x) = \mathcal{P}_1(x)\mu$. In this case, again from Remark~\ref{rem:new2}, 
    \begin{align}\label{eq:new1}
        \sigma^2(x) = \mathcal{P}_1(x)\Sigma - H(x)^2 = H(x)\mu^{-1}\Sigma - H(x)^2 (\geq 0).
    \end{align}
\end{remark}

\section{Proofs of the Main Results}\label{sec:proofs} 

Finally we provide the proofs of all the main results in this section. The proofs heavily use the notations defined in Section~\ref{sec:ldrw} (in particular, Sections~\ref{subsec:model-ldrw} and~\ref{subsec:results-ldrw}), Section~\ref{sec:gerw} (in particular, Sections~\ref{subsec:model-gerw} and~\ref{sec:main-results}) and Section~\ref{sec:sa} (in particular, Sections~\ref{sec:sa-ldrw-gerw} and~\ref{sec:proofs-lemmas}).

\subsection{Proofs of the Main Results of Section~\ref{sec:ldrw}}\label{sec:proofs-ldrw}

We first prove the results from Section~\ref{sec:ldrw}. They will follow from the results of Section~\ref{sec:gerw}, which we prove in Section~\ref{sec:proofs-gerw} independently. Towards this end, we summarize some properties of Model~\ref{eg:ldrw} in the following Proposition, whose proof is immediate.

\begin{proposition}\label{prop:ldrw}
    The one-dimensional generalized elephant random walk can be parameterized as a multidimensional generalized elephant random walk as described in Model~\ref{eg:ldrw}. 
    Thus, in this case, 
    \begin{enumerate}
        \item[a)] $s = d = 1$, $r = 2$, $\boldsymbol{\Pi}^1_2 = \{\{1\},\emptyset\}$, $A = 2$ and $b = -1$.
        \item[b)] ${Y}_1 \stackrel{d}{=} \delta_{{1}}$, %
        $\mu = \mathbb{E}\left({Y}_1\right) = 1$ and $\Sigma = \mathbb{E}\left(Y_1Y_1^{\top}\right) = \mathbb{E}\left(Y_1^2\right) = 1$.
        \item[c)] ${\mathcal{P}_1} = h$, where $h$ is given by \eqref{eq:h} and so $D_s = D_1 = (0,1)$ and ${H}: (0,1) \to (0,1)$ is given by
\begin{align*}
{H}(x) &= \mathcal{P}_1({x}){\mu}^{(\pi_1)} = h(x).
\end{align*} 
\end{enumerate}    
\end{proposition}

\begin{proof}[Proof of Theorem~\ref{thm:slln}]
By Proposition~\ref{prop:ldrw}, we have $\mathbb{E}\left(\|{Y}_1\|^{2}\right) < \infty$ and ${H} = h$. So ${x}_0 = y_0 = (s_0+1)/2$ is the unique fixed point of ${H}$. Also, \eqref{eq3} implies that \eqref{eq33} is satisfied with $D_1 = (0,1)$, $H = h$ and $x_0 = y_0$. With $A = 2$ and $b = -1$, \eqref{eq:asconv1} of Theorem~\ref{thm:slln-gerw} implies \eqref{eq:ldrw-conv}. This completes the proof.
\end{proof}

\begin{proof}[Proof of Theorem~\ref{thm:lil}] 
By Proposition~\ref{prop:ldrw}, we get $s=1$, $\mathbb{E}\left({Y}_1^{2+\epsilon}\right) < \infty$ for all $\epsilon > 0$, ${H} = h$, $A = 2$ and $b = -1$. As the assumptions of Theorem~\ref{thm:slln} are satisfied, %
the assumptions of Theorem~\ref{thm:slln-gerw} hold with $x_0 = y_0 = (s_0+1)/2$. Also $\tau = H'(x_0) = h'((s_0+1)/2) = \eta$ implies Assumption~\ref{clt:assump1} (equivalently, Remark~\ref{rem:clt:assump1-1d}) hold with $\tau \leq 1/2$. In case $\eta < 1/2$, we have $\tau < 1/2$ and 
\begin{align}\label{eq:n1}
   A\sqrt{\frac{x_0\mu^{-1}{\Sigma}-x^2_0}{1-2\tau}} = 2\sqrt{\frac{\frac{s_0+1}{2} -  \left(\frac{s_0+1}{2}\right)^2}{1-2\eta}} = \sqrt{\frac{1-s^2_0}{1-2\eta}}.
\end{align}

Consequently, \eqref{eq:lil-gerw1} of Theorem~\ref{thm:lil-gerw-cric} and \eqref{eq:n1} imply \eqref{eq:lil-ldrw1}. Similarly, in case $\eta = 1/2$, we have $\tau = 1/2$. In this case, as $h$ is twice differentiable at at $(s_0+1)/2$, ${H}$ is also twice differentiable at ${x}_0$. Also,
\begin{align}\label{eq:n2}
   A\sqrt{{x_0\mu^{-1}{\Sigma}-x^2_0}} = 2\sqrt{{\frac{s_0+1}{2} -  \left(\frac{s_0+1}{2}\right)^2}} = \sqrt{{1-s^2_0}}.
\end{align}
Thus, we get \eqref{eq:lil-ldrw2} from \eqref{eq:lil-gerw2} of Theorem~\ref{thm:lil-gerw-cric} and \eqref{eq:n2}, completing the proof.
\end{proof}

\begin{proof}[Proof of Theorem~\ref{thm:clt}] 
As the assumptions of Theorem~\ref{thm:slln} are satisfied, %
the assumptions of Theorem~\ref{thm:slln-gerw} hold with $H = h$ and $x_0 = y_0 = (s_0+1)/2$. Consequently, $\tau = H'(x_0) = h'((s_0+1)/2) = \eta$ and so Assumption~\ref{clt:assump1} (equivalently, remark~\ref{rem:clt:assump1-1d}) are satisfied with $\tau = \eta < 1$ and $\kappa = 1$. By Proposition~\ref{prop:ldrw}, $A = 2$, $b = -1$, $\Sigma = 1$ and $\mathcal{P}_1(x_0) = h(y_0) = y_0 = (s_0+1)/2$. So, 
\[{\Sigma}_0 = \mathcal{P}_1({x}_0){\Sigma}^{(\pi_1)}
   - {x}^2_0 = \frac{s_0+1}{2} -  \left(\frac{s_0+1}{2}\right)^2 = \frac{1 - s_0^2}{4}. 
\]
Thus, in case $\eta < 1/2$, we have $\tau < 1/2$, and in \eqref{eq:sigma1},
\begin{align}\label{eq:m1}
   \Sigma_1 = \int_{0}^{\infty}{\Sigma}_0e^{\left(2\eta - 1\right)u}du = \frac{1 - s_0^2}{4(1-2\eta)}.
\end{align}

Consequently, \eqref{eq:clt-ldrw1} is implied by \eqref{eq:diffconv} of Theorem~\ref{thm:clt-gerw} and \eqref{eq:m1}. When $\eta \geq 1/2$, twice differentiability of $h$ at $(s_0+1)/2$ implies Assumption~\ref{clt:assump2} is 
satisfied. In case $\eta = 1/2$, we have $\tau = 1/2$, and in \eqref{eq:sigma2},
\begin{align}\label{eq:m2}
   \Sigma_2 = \lim_{n \to \infty} \frac{1}{\log n}\int_{0}^{\log n}{\Sigma}_0e^{\left(2\eta - 1\right)u}du = \frac{1 - s_0^2}{4}.
\end{align}
Consequently, \eqref{eq:clt-ldrw2} is implied by \eqref{eq:cricconv} of Theorem~\ref{thm:clt-gerw} and \eqref{eq:m2}. In case $1/2 < \eta < 1$, we have $1/2 < \tau < 1$ and we get \eqref{eq:clt-ldrw3} using \eqref{eq:sdiffconv} of Theorem~\ref{thm:clt-gerw} (equivalently, \eqref{eq:re-eig1} of Remark~\ref{rem:sd}) with $L$ given by \begin{align}\label{eq:L}L = A{\xi} = 2\xi.\end{align} This completes the proof.
\end{proof}

\begin{proof}[Proof of Theorem~\ref{thm:super-dev1}] 
By Proposition~\ref{prop:ldrw}, we have $s=1$, $\mathbb{E}\left({Y}_1^{2+\epsilon}\right) < \infty$ for all $\epsilon > 0$, ${H} = h$, $A = 2$ and $b = -1$. 
As the assumptions of Theorem~\ref{thm:slln} are satisfied, %
the assumptions of Theorem~\ref{thm:slln-gerw} hold with $x_0 = y_0 = (s_0+1)/2$. 
Also $\tau = H'(x_0) = h'((s_0+1)/2) = \eta$ implies Assumption~\ref{clt:assump1} (equivalently, Remark~\ref{rem:clt:assump1-1d}) hold with $1/2 < \tau = \eta < 1$. As $h$ is $(m+1)$-times differentiable around $(s_0+1)/2$, $H$ is $(m+1)$-times differentiable around $x_0$. Also, \begin{align}\label{eq:lim-var}A\sqrt{\frac{x_0\mu^{-1}{\Sigma}-x^2_0}{2\tau-1}} = 2\sqrt{\frac{\frac{s_0+1}{2} -  \left(\frac{s_0+1}{2}\right)^2}{2\eta-1}} = \sqrt{\frac{1-s^2_0}{2\eta-1}}.\end{align} 

Thus, in case $m \geq (\eta - 1/2)/(1-\eta) = (\tau - 1/2)/(1-\tau)$, \eqref{eq:slil-ldrw} (respectively, \eqref{eq:sclt-ldrw}) is obtained from \eqref{eq:slil-gerw} (respectively, \eqref{eq:sclt-gerw}) of Theorem~\ref{thm:super-dev1m} and \eqref{eq:lim-var}, with $m_0 = \lfloor(\tau - 1/2)/(1-\tau)\rfloor = \lfloor(\eta - 1/2)/(1-\eta)\rfloor$, $L$ given by \eqref{eq:L} and $\{\beta_j\}_{j=1}^{m+1}$ is recursively given by \begin{align}\label{eq:beta}\beta_1 = \frac{b_1}{A^0} = 1, &\beta_{j+1} = \frac{b_{j+1}}{A^j} =-\frac{1}{A^jj(1-\tau)}\sum_{p=2}^{j+1}\frac{H^{(p)}(x_0)}{p!}\sum_{(c_1, \ldots, c_p) \in \mathcal{P}_{p,j+1}}\nu_{(c_1, \ldots, c_p)}b_{c_1}\ldots b_{c_p} \\&= -\frac{1}{j(1-\eta)}\sum_{p=2}^{j+1}\frac{h^{(p)}\left(\frac{s_0+1}{2}\right)}{2^{p-1}p!}\sum_{(c_1, \ldots, c_p) \in \mathcal{P}_{p,j+1}}\nu_{(c_1, \ldots, c_p)}\beta_{c_1}\ldots \beta_{c_p}, j = 1, \ldots, m.\nonumber\end{align} 
where, for $1 \leq p \leq t$, $\mathcal{P}_{p,t}$ and $\nu_{(c_1, \ldots, c_p)}$ are given by \eqref{eq:combn}.
Similarly, in case $m < (\eta - 1/2)/(1-\eta) = (\tau - 1/2)/(1-\tau)$, \eqref{eq:s-ldrw} is obtained from \eqref{eq:s-gerw} of Theorem~\ref{thm:super-dev1m}, with $L$ given by \eqref{eq:L} and $\beta_j$, $j = 0, \ldots, m$ given by \eqref{eq:beta}. 
\end{proof}

\subsection{Proofs of the Main Results of Section~\ref{sec:gerw}}\label{sec:proofs-gerw}

We now prove the main results of the model given in Section~\ref{sec:gerw}. We prove these results independently of the results from Section~\ref{sec:ldrw}.

\begin{proof}[Proof of Theorem~\ref{thm:slln-gerw}]
Note that there exists a unique root $\boldsymbol{x}_0 \in {D}_s$ of the equation $\boldsymbol{\gamma}(\boldsymbol{x}) = \boldsymbol{0}$ and \eqref{eq33} implies that for all 
$\epsilon > 0$ 
\begin{align}\label{eq-relaxcontpf2}
     \sup_{\epsilon \leq \|\boldsymbol{x} - \boldsymbol{x}_0\| \leq \frac{1}{\epsilon}}  \left(\boldsymbol{x} - \boldsymbol{x}_0\right)^{\top}\left(-\boldsymbol{\gamma}(\boldsymbol{x})\right) < 0.
\end{align}
In view of \eqref{eq-sa2}, \eqref{eq-relaxcontpf2} and Lemma~\ref{lem-noiseanddrift2}, $\{\boldsymbol{\Gamma}_n\}_{n \geq 1}$ satisfies all the assumptions of Theorem 1 of Section 5 of \cite{MR1082341} (Part II), using which we conclude that almost surely $\boldsymbol{\Gamma}_{n}$ converges to $\boldsymbol{x}_0$. Consequently, we get that almost surely ${\boldsymbol{S}_n}/{n} = \boldsymbol{A}\boldsymbol{\Gamma}_{n} + \boldsymbol{b}$ converges to $\boldsymbol{A}\boldsymbol{x}_0 + \boldsymbol{b}$, completing the proof.
\end{proof}

\begin{proof}[Proof of Theorem~\ref{thm:clt-gerw}] Since the assumptions of Theorem~\ref{thm:slln-gerw} hold, almost surely $\boldsymbol{\Gamma}_{n}$ converges to $\boldsymbol{x}_0$ which is as described in Theorem~\ref{thm:slln-gerw}. Note that Assumption~\ref{clt:assump1} imply that $\boldsymbol{\gamma}$ is differentiable at $\boldsymbol{x}_0$ and all the eigenvalues of $J_{\boldsymbol{\gamma}}\left(\boldsymbol{x}_0\right) = \mathbb{I}_s - J_{\boldsymbol{H}}\left(\boldsymbol{x}_0\right)$ have positive real parts. Thus,
\begin{align*}
       \rho := \min \left\{\operatorname{Re}\left(\zeta\right) : \zeta \in \{1 - \lambda_1, \ldots, 1 - \lambda_l\} \right\} = 1 - \tau > 0.
\end{align*}
So, $\boldsymbol{\gamma}$ and its zero $\boldsymbol{x}_0$ satisfy Assumption 2.1 of \cite{zhang2016central}. 
From \eqref{eq:lem-Sigma} and \eqref{eq:lem-Sigmaconv}, we get
\begin{align}\label{eq:Sigma-proof}
       \frac{1}{n}\sum_{k=1}^{n}\mathbb{E}\left(\boldsymbol{e}_{k+1}\boldsymbol{e}^{\top}_{k+1} \mid \boldsymbol{\Gamma}_{k}\right) \stackrel{a.s.}{\to} \boldsymbol{\Sigma}\left(\boldsymbol{x}_0\right)
\end{align}

As $\boldsymbol{H}\left(\boldsymbol{x}_0\right) = \boldsymbol{x}_0$, we have that $\boldsymbol{\Sigma}\left(\boldsymbol{x}_0\right) = \boldsymbol{\Sigma}_0$. Thus, \eqref{eq:lind} and \eqref{eq:Sigma-proof} imply that $\{\boldsymbol{e}_{n}\}_{n\geq2}$ satisfies Assumption 2.3 of \cite{zhang2016central}. In case $\tau < 1/2$, we have $\rho > 1/2$. Hence, in view of \eqref{eq-sa2}, $\{\boldsymbol{\Gamma}_n\}_{n \geq 1}$ satisfies all the assumptions of Theorem 2.3 of \cite{zhang2016central} (with the remainder sequence $\left\{\boldsymbol{r}_{n+1}\right\}_{n \geq 1}$ of \cite{zhang2016central} being sequence of all zeros), which implies, 
\begin{align}\label{eq:diff-conv-proof}
    \sqrt{n}\left(\boldsymbol{\Gamma}_{n} - \boldsymbol{x}_0\right) \stackrel{d}{\to} N\left(0,\boldsymbol{\Sigma}_1\right).
\end{align}
When $\tau \geq 1/2$, Assumption~\ref{clt:assump2} imply  
\begin{align*}
   \boldsymbol{\gamma}(\boldsymbol{x}) = \boldsymbol{\gamma}(\boldsymbol{x}_0) + J_{\boldsymbol{\gamma}}\left(\boldsymbol{x}_0\right)%
   \left(\boldsymbol{x} - \boldsymbol{x}_0\right)  + o\left(\|\boldsymbol{x} - \boldsymbol{x}_0\|^{1+\delta}\right) \quad \text{as } \boldsymbol{x} \to \boldsymbol{x}_0,
\end{align*}
and thus Assumption 2.2 of \cite{zhang2016central} is satisfied.
In case $\tau = 1/2$, we have $\rho = 1/2$. 
Hence, in view of \eqref{eq-sa2}, $\{\boldsymbol{\Gamma}_n\}_{n \geq 1}$ satisfies all the assumptions of Theorem 2.1 of \cite{zhang2016central} (with $\left\{\boldsymbol{r}_{n+1}\right\}_{n \geq 1}$ of \cite{zhang2016central} being sequence of all zeros), which implies, 
\begin{align}\label{eq:cric-conv-proof}
    \frac{\sqrt{n}}{(\log n)^{\kappa - 1/2}}\left(\boldsymbol{\Gamma}_{n} - \boldsymbol{x}_0\right) \stackrel{d}{\to} N\left(0,\boldsymbol{\Sigma}_2\right).
\end{align}

Now observe that \eqref{eq:lem-Sigma} and \eqref{eq:lem-Sigmaconv} imply that 
\[
       \sum_{k=1}^{n}\mathbb{E}\left(\boldsymbol{e}_{k+1}\boldsymbol{e}^{\top}_{k+1} \mid \boldsymbol{\Gamma}_{k}\right) = \operatorname{O}\left(n\right) \quad \text{a.s.},
    \]
and, consequently, $\{\boldsymbol{e}_{n}\}_{n\geq2}$ satisfies (2.10) of \cite{zhang2016central}. In case $\tau > 1/2$,  we have $\rho < 1/2$. Hence, in view of \eqref{eq-sa2}, $\{\boldsymbol{\Gamma}_n\}_{n \geq 1}$ satisfies all the assumptions of Theorem 2.2 of \cite{zhang2016central} (with $\left\{\boldsymbol{r}_{n+1}\right\}_{n \geq 1}$ of \cite{zhang2016central} being sequence of all zeros). Hence we conclude there exist random variables ${\xi}_1, {\xi}_2, \ldots$ such that
\begin{align}\label{eq:sdiff-conv-proof}
    \frac{n^{1-\tau}}{(\log n)^{\kappa - 1}}\left(\boldsymbol{\Gamma}_{n} - \boldsymbol{x}_0\right) - \sum_{j : \operatorname{Re}\left(\lambda_j\right) = \tau, \kappa_j = \kappa}e^{\iota\operatorname{Im}\left(\lambda_j\right)\log n}{\xi}_j\boldsymbol{R}^*_{j} \stackrel{a.s.}{\to} 0. %
\end{align}
As ${\boldsymbol{S}_n}/{n} = \boldsymbol{A}\boldsymbol{\Gamma}_{n} + \boldsymbol{b}$, \eqref{eq:diff-conv-proof} (respectively, \eqref{eq:cric-conv-proof}, \eqref{eq:sdiff-conv-proof}) implies \eqref{eq:diffconv} (respectively, \eqref{eq:cricconv}, \eqref{eq:sdiffconv}). This completes the proof.
\end{proof}

\begin{proof}[{Proof of Theorem~\ref{thm:lil-gerw-cric}}] Since the assumptions of Theorem~\ref{thm:slln-gerw} hold, almost surely ${\Gamma}_{n}$ converges to ${x}_0$ which is as described in Theorem~\ref{thm:slln-gerw}. Also, Assumption~\ref{clt:assump1} (in particular, Remark~\ref{rem:clt:assump1-1d}) imply that ${\gamma}$ is differentiable in a neighborhood of ${x}_0$ and ${\gamma}'({x}_0) = 1 - \tau \geq 1/2$. Moreover, as $H(x_0) = x_0$, from \eqref{eq:new1},%
it follows that
\begin{align}\label{eq:sigma-proof}
    \sigma^2\left({x}_0\right) &= %
    x_0\mu^{-1}\Sigma - x_0^2. %
\end{align} So, in case of $\tau < 1/2$, we have ${\gamma}'({x}_0) > 1/2$. Therefore, in view of \eqref{eq-sa2} and Lemma~\ref{lem-noiseanddrift2}, %
$\{{\Gamma}_n\}_{n \geq 1}$ satisfies all the assumptions of part a) of Theorem~\ref{thm:sa1}. As ${{S}_n}/{n} = {A}{\Gamma}_{n} + {b}$, \eqref{eq:lil-gerw1} follows from part a) of Theorem~\ref{thm:sa1}.

In case of $\tau = 1/2$, we have ${\gamma}'({x}_0) = 1/2$. Also, as ${H}$ is twice differentiable at ${x}_0$ in this case, so is $\gamma$. Hence, in view of \eqref{eq-sa2} and Lemma~\ref{lem-noiseanddrift2}, %
$\{{\Gamma}_n\}_{n \geq 1}$ satisfies all the assumptions of part b) of Theorem~\ref{thm:sa1}, which, along with the fact that ${{S}_n}/{n} = {A}{\Gamma}_{n} + {b}$, imply \eqref{eq:lil-gerw2}.
\end{proof}

\begin{proof}[{Proof of Theorem~\ref{thm:super-dev1m}}] Since the assumptions of Theorem~\ref{thm:slln-gerw} hold, almost surely ${\Gamma}_{n}$ converges to ${x}_0$ which is as described in Theorem~\ref{thm:slln-gerw}. Assumption~\ref{clt:assump1} (in particular, Remark~\ref{rem:clt:assump1-1d}) imply that ${\gamma}$ is differentiable around ${x}_0$ and $0 < {\gamma}'({x}_0) = 1 - \tau < 1/2$. Also, $\gamma$ is $(m+1)$-times differentiable around $x_0$. If $m \geq (\tau-1/2)/(1-\tau)$, there exists $l \in \mathbb{N}\cup\{0\}$ such that 

\[ \frac{\tau-1/2}{1-\tau} + l \leq m < \frac{\tau-1/2}{1-\tau} + l + 1
\] In case 
\[\frac{\tau-1/2}{1-\tau} + l < m < \frac{\tau-1/2}{1-\tau} + l + 1, \text{ we have }  \frac{1}{2(m-l + 1)}< {\gamma}'({x}_0) < \frac{1}{2(m-l)}.\]

In view of \eqref{eq-sa2} and Lemma~\ref{lem-noiseanddrift2}, %
$\{{\Gamma}_n\}_{n \geq 1}$ satisfies all the assumptions of Theorem~\ref{thm:A2} %
(with $k = m-l-1$). Thus \eqref{eq:suppwc} implies that
\begin{align}\label{eq:sdiffclt-pf1}
    \sqrt{n}\left(\left({\Gamma}_{n} - {x}_0\right) - \sum_{j = 0}^{m-l-1} b_{j+1}\left(\frac{{\xi}}{n^{{\gamma}'({x}_0)}}\right)^{j+1}\right) \stackrel{d}{\to} {N}\left(0, \frac{{\sigma}^2(x_0)}{1 - 2{\gamma}'({x}_0)}\right),
\end{align}
and \eqref{eq:supplil} implies that %
almost surely,
\begin{align}\label{eq:sdifflil-pf1}
    &\limsup_{n\to\infty}\sqrt{\frac{n}{2\log \log n}}\left(\left({\Gamma}_{n} - {x}_0\right) - \sum_{j = 0}^{m-l-1} b_{j+1}\left(\frac{{\xi}}{n^{{\gamma}'({x}_0)}}\right)^{j+1}
    \right) \nonumber\\
    =-&\liminf_{n\to\infty}\sqrt{\frac{n}{2\log \log n}}\left(\left({\Gamma}_{n} - {x}_0\right) - \sum_{j = 0}^{m-l-1} b_{j+1}\left(\frac{{\xi}}{n^{{\gamma}'({x}_0)}}\right)^{j+1} \right)
   = \sqrt{\frac{{\sigma}^2(x_0)}{1 - 2{\gamma}'({x}_0)}}, 
\end{align}
where, $\sigma^2(x_0)$ is given by \eqref{eq:sigma-proof}, and %
the sequence $\{b_j\}_{j=1}^{m+1}$ is recursively given by 
\begin{align}\label{eq:coeff}b_1 = 1, \quad b_{j+1} = \frac{1}{j\gamma'(x_0)}\sum_{p=2}^{j+1}\frac{\gamma^{(p)}(x_0)}{p!}\sum_{(c_1, \ldots, c_p) \in \mathcal{P}_{p,j+1}}\nu_{(c_1, \ldots, c_p)}b_{c_1}\ldots b_{c_p},\end{align}
where, for $1 \leq p \leq t$, $\mathcal{P}_{p,t}$ and $\nu_{(c_1, \ldots, c_p)}$ are given by \eqref{eq:combn}.
In case 
\[ m = \frac{\tau-1/2}{1-\tau} + l, \text{ we have } {\gamma}'({x}_0) = \frac{1}{2(m-l+1)}.\] 

Thus, using \eqref{eq-sa2} and Lemma~\ref{lem-noiseanddrift2} %
we conclude that $\{{\Gamma}_n\}_{n \geq 1}$ satisfies all the assumptions of Theorem~\ref{thm:A2} %
(with $k = m-l$). %
Thus \eqref{eq:suppwc} implies that
\begin{align}\label{eq:sdiffclt-pf2}
    \sqrt{n}\left(\left({\Gamma}_{n} - {x}_0\right) - \sum_{j = 0}^{m-l} b_{j+1}\left(\frac{{\xi}}{n^{{\gamma}'({x}_0)}}\right)^{j+1} \right)
    \stackrel{d}{\to} {N}\left(0, \frac{{\sigma}^2(x_0)}{1-2\gamma'(x_0)}\right),
\end{align}
and \eqref{eq:supplil} implies that, almost surely,
\begin{align}\label{eq:sdifflil-pf2}
    \limsup_{n\to\infty}\sqrt{\frac{n}{2\log \log n}}&\left(\left({\Gamma}_{n} - {x}_0\right) - \sum_{j = 0}^{m-l} b_{j+1}\left(\frac{{\xi}}{n^{{\gamma}'({x}_0)}}\right)^{j+1}
    \right) \nonumber\\
    =-\liminf_{n\to\infty}\sqrt{\frac{n}{2\log \log n}}&\left(\left({\Gamma}_{n} - {x}_0\right) - \sum_{j = 0}^{m-l} b_{j+1}\left(\frac{{\xi}}{n^{{\gamma}'({x}_0)}}\right)^{j+1} \right)
    = \sqrt{\frac{{\sigma}^2(x_0)}{1-2\gamma'(x_0)}}.\hspace{0.75cm}& 
\end{align}
Finally, in case \[m < \frac{\tau-1/2}{1-\tau}, \text{ we have } 0 < {\gamma}'({x}_0) < \frac{1}{2(m+1)}.\] From \eqref{eq-sa2} and Lemma~\ref{lem-noiseanddrift2} %
we see that $\{{\Gamma}_n\}_{n \geq 1}$ satisfies all the assumptions of Theorem~\ref{thm:A1} %
(with $k = m$). Thus \eqref{eq:nn1} implies that%
\begin{align}\label{eq:sconv-proof}
    n^{\gamma'(x_0)(m+1)}\left(\left({\Gamma}_{n} - {x}_0\right) - \sum_{j = 0}^{m-1} b_{j+1}\left(\frac{{\xi}}{n^{{\gamma}'({x}_0)}}\right)^{j+1} \right)
    \stackrel{a.s.}{\to}  b_{m+1}{\xi}^{m+1}  .
\end{align}
As ${{S}_n}/{n} = {A}{\Gamma}_{n} + {b}$ and $\gamma(x) = x - H(x)$, 
\eqref{eq:sdifflil-pf1} and \eqref{eq:sdifflil-pf2} (respectively, \eqref{eq:sdiffclt-pf1} and \eqref{eq:sdiffclt-pf2}, \eqref{eq:sconv-proof}) imply \eqref{eq:slil-gerw} (respectively, \eqref{eq:sclt-gerw}, \eqref{eq:s-gerw}). This completes the proof.
\end{proof}

\section*{Acknowledgements}

We are grateful to Jean Bertoin and Rahul Roy for providing us with the references \cite{ldpir} and \cite{wu2012asymptotics}, respectively.

\providecommand{\bysame}{\leavevmode\hbox to3em{\hrulefill}\thinspace}
\providecommand{\MR}{\relax\ifhmode\unskip\space\fi MR }
\providecommand{\MRhref}[2]{%
  \href{http://www.ams.org/mathscinet-getitem?mr=#1}{#2}
}
\providecommand{\href}[2]{#2}

\end{document}